\documentclass[11pt]{article} 

\usepackage[utf8]{inputenc} 

\usepackage{cancel}

\usepackage{geometry} 
\usepackage{graphicx} 
\usepackage[english]{babel}

\usepackage{booktabs} 
\usepackage{array} 
\usepackage{paralist} 
\usepackage{verbatim} 
\usepackage{subfig} 
\usepackage{amssymb}
\usepackage{bbm}

\usepackage{amsmath}
\usepackage{amsthm}
\usepackage{makecell}

\usepackage{epigraph}

\usepackage{hyperref}
\hypersetup{
    colorlinks=true,
    citecolor=blue
}

\usepackage{pdflscape}
\usepackage{natbib}
\usepackage{algorithm2e}

\usepackage{wrapfig}
\usepackage{subfig}
\usepackage[capitalize,noabbrev]{cleveref}
\usepackage{authblk}
\usepackage{changepage}

\usepackage{tikz-cd}
\usepackage{xcolor}

\renewcommand{\le}{\leqslant}

\renewcommand{\ge}{\geqslant}

\newcommand{\id}{\mathbbmss 1}

\newcommand {\matl}{\left[ \begin{matrix}}
\newcommand {\matr}{\end{matrix}\right]}
\newcommand {\Exp}{ \mathbb E }

\renewcommand {\Pr}{ \mathbb P }

\newcommand{\blamb}{\boldsymbol{\lambda}}
\newcommand{\bnu}{\boldsymbol{\nu}}
\newcommand{\bbeta}{\boldsymbol{\beta}}
\newcommand{\bgam}{\boldsymbol{\gamma}}
\newcommand{\lKT}{\lambda^{\mathsf{KT}}}
\newcommand{\lkelly}{\lambda^{\mathsf{Kelly}}}
\newcommand{\lgr}{\lambda^{\mathsf{GRAPA}}}
\newcommand{\lKL}{\lambda^{\mathsf{KL}}}
\newcommand{\lagr}{\lambda^{\mathsf{aGRAPA}}}

\newcommand {\Var}{\mathbf{Var}}

\newcommand{\kl}{D_{\operatorname{KL}}}
\newcommand{\klinf}{\operatorname{KL}_{\inf}}

\renewcommand{\d}{\mathrm{d}}
\newcommand{\iid}{\overset{\mathrm{iid}}{\sim}}

\newcommand{\cD}{\mathcal{D}}

\newcommand{\CI}{\operatorname{CI}}

\DeclareMathAlphabet{\mathbbmsl}{U}{bbm}{m}{sl}

\DeclareMathOperator*{\argmax}{arg\,max}

\newtheorem{theorem}{Theorem}[section]
\newtheorem{definition}[theorem]{Definition}

\newtheorem{fact}{Fact}
\newtheorem{proposition}[theorem]{Proposition}

\newtheorem{corollary}[theorem]{Corollary}
\newtheorem{lemma}[theorem]{Lemma}
\theoremstyle{definition}\newtheorem{remark}[theorem]{Remark}
\newtheorem{example}[theorem]{Example}

\Crefname{fact}{Fact}{Facts}
\Crefname{assumption}{Assumption}{Assumptions}

\title{Almost sure null bankruptcy of 
testing-by-betting strategies\footnote{COLT 2026}}

\author[1]{Hongjian Wang}
\author[2]{Shubhada Agrawal}
\author[3]{Aaditya Ramdas}
\affil[1]{Carnegie Mellon University \href{mailto:hjnwang@cmu.edu}{\texttt{hjnwang@cmu.edu}}}
\affil[2]{Indian Institute of Science  \href{mailto:shubhada@iisc.ac.in}{\texttt{shubhada@iisc.ac.in}}}
\affil[3]{Carnegie Mellon University \href{mailto:aramdas@cmu.edu}{\texttt{aramdas@cmu.edu}}}

\date{\today}

\begin{document}

\maketitle

\begin{abstract}
    The bounded mean betting procedure serves as a crucial interface between the domains of (1) sequential, anytime-valid statistical inference, and (2) online learning and portfolio selection algorithms. While recent work in both domains has established the exponential wealth growth of numerous betting strategies under any alternative distribution, the tightness of the inverted confidence sets, and the pathwise minimax regret bounds, little has been studied regarding the asymptotics of these strategies under the null hypothesis. Under the null, a strategy 
induces a wealth martingale converging to some random variable that can be zero (bankrupt) or non-zero (non-bankrupt, e.g.\ when it eventually stops betting). 
In this paper, we show the conceptually intuitive but technically nontrivial fact that these strategies (universal portfolio, Krichevsky-Trofimov, GRAPA, hedging, etc.)\ all go bankrupt with probability one, under any non-degenerate null distribution. Part of our analysis is based on the subtle almost sure divergence of various sums of $\sum_n O_p(n^{-1})$ type, 
a result of independent interest.
We also demonstrate the necessity of null bankruptcy by showing that non-bankrupt strategies are all improvable in some sense.
Our results significantly deepen our understanding of these betting strategies as they qualify their behavior on ``almost all paths'', whereas previous results are usually on ``all paths'' (e.g.\ regret bounds) or ``most paths'' (e.g.\ concentration inequalities and confidence sets).
  
\end{abstract}


\section{Introduction}\label{sec:intro}

We consider the problem of testing the mean of i.i.d.\ random variables taking values in $[0,1]$ via betting, studied by various authors including \cite{shafer2005probability,shafer2019game,shafer2021testing,waudby2024estimating,orabona2023tight,voracek2025starbets}. Let $(X_n) \iid P$ where $P$ is some distribution on $[0,1]$ with unknown mean $\mu(P) = \Exp X_1 \in [0,1]$, and let the null hypothesis be $H_0: \mu(P) = m$. To set the stage for our upcoming discourse, we recall that the \emph{betting wealth process} that sequentially tests $H_0$ with
a \emph{fixed} bet fraction $\lambda \in [-\frac{1}{1-m}, \frac{1}{m}]$ is
\begin{equation}\label{eqn:bet-fixed}
    W_n^\lambda := \prod_{k=1}^n (1 + \lambda(X_k - m) ).
\end{equation}
$(W_n^\lambda)$ is interpreted as follows. The statistician starts with unit wealth, and bets round by round on the outcomes of $X_1,X_2,\dots$, where $x \in \mathbb R$ units of bet on $X_k$ placed before the revelation of $X_k$ leads to $x\cdot (X_k - m)$ units of profit (if positive) or loss (if negative). Each round, the statistician bets a fixed fraction $\lambda$ of the current wealth.
$W_n^\lambda$ is therefore the statistician's wealth after $n$ such rounds. Clearly, $(W^\lambda_n)$ forms
a nonnegative martingale under any distribution $P$ on $[0,1]$ that satisfies the null $H_0: \mu(P)=m$ (``null distribution'' henceforth). Therefore, a key doctrine of game-theoretic statistics \citep{ramdas2023game} states that $(W^\lambda_n)$ quantifies the accumulation of evidence carried by the observations $(X_n)$ against the null hypothesis $H_0: \mu(P) = m$ in favor of the alternative hypothesis $H_1: \mu(P) \neq m$.


The fixed-fraction betting strategy \eqref{eqn:bet-fixed}, however, fails to be universally powerful in the sense that there always exists some alternative distribution (i.e.\ $P$ such that $\mu(P) \neq m$) under which $W_n^\lambda \to 0$ almost surely. To see that,  when $\lambda > 0$, under any non-degenerate $P$ such that $\mu(P) < m$, $\log W_n^\lambda \le \sum_{k=1}^n \lambda(X_k - m) \to -\infty$ and consequently $W_n^\lambda \to 0$. To overcome this, three (overlapping) classes of betting strategies derived from \eqref{eqn:bet-fixed} are 
commonly employed.


\paragraph{Predictable plug-in.} Let $\blamb = (\lambda_n)_{n \ge 1}$ be a $[-\frac{1}{1-m}, \frac{1}{m}]$-valued stochastic process that is predictable with respect to the filtration $\sigma(X_1,\dots, X_n)$. The wealth process corresponding to betting $\lambda_k$ fraction of wealth on the upcoming $X_k$ reads
\begin{equation}\label{eqn:predictable}
    W_n^{\blamb} := \prod_{k=1}^n (1 + \lambda_k(X_k - m) ).
\end{equation}
$(W_n^{\blamb})$ remains a nonnegative martingale under $H_0$. Usually, the value of $\lambda_n$ is picked based on $X_1,\dots, X_{n-1}$ to match the sign of the ``reality-house discrepancy'' $\mu(P) - m$. The simplest idea is to use the empirical mean of $X_1-m, \dots, X_{n-1}-m$, leading to the Krichevsky-Trofimov-type (KT) bettor \citep{krichevsky1981performance,orabona2016coin}:
\begin{equation}\label{eqn:kt-lambda}
    \lKT_n := \frac{1/2 + \sum_{k=1}^{n-1}(X_k - m)}{Cn},
\end{equation}
where $C \ge {m(1-m)}$ is an appropriate constant. The KT bettor with $C={m(1-m)}$ is the ``standard'' KT bettor in the literature, where $\lKT_n$ estimates $\frac{\mu(P) - m}{m(1-m)}$, a quantity that equals the \emph{Kelly-optimal} bet fraction maximizing the expected log-payoff per round
\begin{equation}\label{eqn:kelly-optimal}
   \lkelly = \argmax_{\lambda\in [-\frac{1}{1-m}, \frac{1}{m}]} \bigg\{\Exp_P(\log(1+\lambda(X_1 - m)))\bigg\}
\end{equation}
when the data-generating distribution $P$ is the Bernoulli coin-toss $\operatorname{bernoulli}(\mu(P))$.
In general, for non-Bernoulli $P$, the Kelly-optimal bet fraction \eqref{eqn:kelly-optimal} needs to be estimated instead by its natural M-estimator:
\begin{equation}\label{eqn:grapa-lambda}
    \lgr_n := \argmax_{ \lambda \in [-\frac{1}{1-m}, \frac{1}{m}]} \left\{ \frac{1}{n-1}\sum_{k=1}^{n-1} \log(1+\lambda(X_k - m)) \right\} = \argmax_{ \lambda \in [-\frac{1}{1-m}, \frac{1}{m}]} W_{n-1}^\lambda. 
\end{equation}
That is, $\lgr_n$ is the hindsight optimal fixed bet fraction after $n-1$ rounds (a.k.a.\ ``follow-the-leader''), and is therefore named the ``growth rate adaptive to the particular alternative'' (GRAPA) bettor by \cite{waudby2024estimating}. GRAPA requires an optimization procedure that may be computationally undesirable, therefore \cite{waudby2024estimating} also propose the approximate GRAPA (aGRAPA) bettor, which we discuss later.

The method of predictable plugin betting \eqref{eqn:predictable} is the widest class of testing by betting strategies, among which we name KT and GRAPA here and include the additional aGRAPA in \cref{sec:bankrupt-pred}. The other two classes of betting strategies are both subclasses of predictable plugin betting. In fact, \citet[Proposition 2]{waudby2024estimating} prove that \emph{all} nonnegative martingales under $H_0:\mu(P)=m$ are representable by \eqref{eqn:predictable} for some predictable sequence $(\lambda_n)$. Nonetheless, we introduce the other two classes of betting strategies, since it is often easier to analyze the property of a betting strategy if it falls in to one of these two subclasses.

\paragraph{Mixture.} Let $\pi$ be a probability measure on $[-\frac{1}{1-m},\frac{1}{m}]$. The wealth process corresponding to the portfolio consisting of the (possibly continuous) collection of different fixed-fraction bets $\{(W_n^\lambda) : \lambda \in \operatorname{supp}(\pi) \}$, weighted by $\lambda \sim \pi$, is
\begin{equation}\label{eqn:mixture-wealth}
    W_n^\pi := \int W_n^\lambda \, \pi(\d \lambda) = \int \left\{\prod_{k=1}^n (1 + \lambda(X_k - \mu))  \right\} \pi(\d \lambda).
\end{equation}
It is easy to verify that the mixture wealth \eqref{eqn:mixture-wealth} can be represented as an instance of predictable plug-in \eqref{eqn:predictable} with the predictable fraction sequence $\lambda_k^\pi = \frac{\int \lambda W_{k-1}^\lambda  \pi(\d \lambda) }{ \int W_{k-1}^\lambda  \pi(\d \lambda) }$, and is also a nonnegative martingale under $H_0$.
Some choices of mixture measure $\pi$ that are continuously supported on $[-1,1]$ trace back to \cite{robbins1970statistical} in the context of subGaussian mean testing, as well as to \cite{cover1991universal,cover2002universal} in the context of portfolio selection (``universal portfolios''). Via these mixture measures, \cite{orabona2023tight} establish a link from regret bounds to confidence sequences for $\mu(P)$; and
we shall soon see that the specific expressions of these measures do not matter for our current paper.

\paragraph{Predictable hedging.} Finally, there is a class of strategies that resemble in form both predictable plug-in and mixture strategies. Let $\blamb = (\lambda_n)$ be a $[0, \min(\frac{1}{1-m}, \frac{1}{m} ) )$-valued predictable process. The hedged betting wealth process based on $\blamb$ is
\begin{equation}\label{eqn:hedged}
    W_n^{\pm\blamb} := \frac{1}{2} \prod_{k=1}^n (1 + \lambda_k(X_k - m) ) + \frac{1}{2} \prod_{k=1}^n (1 - \lambda_k(X_k - m) ),
\end{equation}
and is also a nonnegative martingale under $H_0$.
That is, one takes a two-point mixture among the strategies $W_n^{\blamb}$ and $W_n^{-\blamb}$. The fraction sequence $\blamb$ here is often taken at the decreasing rate $\lambda_n \asymp n^{-1/2}$ or $(n\log n)^{-1/2}$, choices demonstrated by \cite{waudby2024estimating,shekhar2023near} to enable optimal $(1-\alpha)$-confidence sets for $\mu(P)$. A specific choice recommended by \citet[Equation (26)]{waudby2024estimating} reads
\begin{equation}\label{eqn:hedging-rate}
    \lambda_n^{\mathsf{PrH}} = \left( -\frac{C}{1-m} \right) \vee \sqrt{\frac{2\log(2/\alpha)}{\widehat \sigma_{n-1}^2 n \log(n+1)}}
    \wedge \frac{C}{m},
\end{equation}
where $\widehat \sigma_{n-1}^2$ is an appropriate consistent variance estimator from $X_1,\dots, X_{n-1}$, and $C\in(0,1)$ some clipping constant.

\vspace{1em}

Many variations of these betting strategies exist.
To quote from \cite{waudby2024estimating}, 
``each of these betting strategies have their respective benefits, whether computational, conceptual, or
statistical''. We refer the reader to the original work by \cite{waudby2024estimating} and \cite{orabona2023tight} as well as the papers cited therein for more discussions on betting strategies for the bounded mean problem.

A crucial shared property that separates all these involved strategies from the na\"ive fixed-fraction strategy \eqref{eqn:bet-fixed} is that they are all \emph{universally power-one}. That is, for \emph{any} alternative distribution $P$ (i.e.\ $P$ on $[0,1]$ with mean $\mu(P)\neq m$), the wealth process $(W_n)$ of these strategies always grows to infinitely almost surely: $P(W_n \to \infty)=1$. 
Very often, the wealth growth happens almost surely at an exponential rate: $P(\liminf n^{-1}\log W_n > 0) = 1$. Indeed, it is not hard to see that the KT bettor with $C = 2$, the GRAPA bettor, and the mixture bettor with continuous $\pi$ on $[-1,1]$ are all exponentially powerful in this sense.

Additionally, as is frequently alluded to in our introduction to them, many of these strategies are adapted from the online learning and portfolio selection literature where there is neither hypothesis testing framework nor probability distribution assumed on the observation sequence $(X_n)$; they enjoy sharp pathwise regret bounds $R_n := \sup_{\lambda} (\log W_n^\lambda) - \log W_n$ at the rate of $R_n \lesssim \log n$ on \emph{every} sample path $(X_n) \in [0,1]^\infty$. Many, on the other hand, lead to $(1-\alpha)$-confidence sequences for $\mu(P)$ by inversion $\CI_n = \{ m : W_n(m) \le 1/\alpha \}$ (where $W_n(m)$ is the wealth testing $H_0: \mu(P) = m$) of optimal size.

\textbf{The key contribution} of our current paper is that, these well-designed betting strategies which oftentimes enjoy exponential almost sure growth rates under alternative distributions, or logarithmic regret bounds in the pathwise sense, or minimax optimal confidence sequences under inversion, would \emph{all fall into bankruptcy \textbf{with probability one} under any non-degenerate null distribution}.
That is, as long as $\mu(P) = m$ and $P$ is not a point mass at $m$,
\begin{equation}\label{eqn:bankruptcy}
    P( W_n \to 0 ) = 1
\end{equation}
for all these wealth processes $(W_n)$. Further, we demonstrate that the null bankruptcy phenomenon \eqref{eqn:bankruptcy} does \emph{not} happen on some suboptimal variants of these betting strategies, which suggests that null bankruptcy is a fundamental property of ``good'' strategies.

Our results that $P( W_n \to 0 ) = 1$ under any non-degenerate null $P$ complete the picture regarding the behavior of these betting strategies in the asymptotic regime, complementary to the fact that $P(W_n \to \infty) = 1$ under any alternative $P$  mentioned earlier. They also deepen our prior understanding of these wealth processes under the null: we previously were only aware of the fact that $(W_n)$ is a nonnegative martingale under any null $P$, and therefore (1) it must converge a.s.\ to \emph{some random variable}, and (2) it satisfies the nonasymptotic Ville's inequality $P( \sup W_n < x) \ge 1- x^{-1}$, i.e.\ it is unlikely to accumulate wealth more than $x$, \emph{for $x$ (much) larger than 1}.

The null bankruptcy of the simple and less powerful fixed-fraction betting strategy \eqref{eqn:bet-fixed}, we note, is well-understood and known as the ``gambler's ruin'' in elementary probability textbooks. The null bankruptcy of these involved, power-one, pathwise minimal-regret strategies that we prove in this paper, in contrary, requires applying and devising of some insightful results in asymptotic probability, and oftentimes happens much more subtly. At the end of this paper, we also 
show that a subclass of strategies that do not go bankrupt under the null are all ``improvable'' in some sense. 

\paragraph{Notation.} We shall frequently employ the asymptotic notations in its both pathwise and in-probability usages. Let $(X_n)$ be a sequence of random variables and $(a_n)$ a sequence of nonrandom positive numbers. We say $X_n = O_{a.s.}(a_n)$ etc.\ if the pathwise event $X_n=O(a_n)$ happens with probability 1.
We recall $X_n = O_p(a_n)$ if for any $\varepsilon > 0$, there exists $M > 0$ such that
\begin{equation}
  \Pr(|X_n| \le M a_n) \ge 1- \varepsilon \quad \text{for all but finitely many }n.
\end{equation}
Similarly, we say $X_n = \Omega_p(a_n)$ if for any $\varepsilon > 0$ there exists $\delta > 0$ such that
\begin{equation}
  \Pr(|X_n| \ge \delta a_n) \ge 1- \varepsilon \quad \text{for all but finitely many }n.
\end{equation}
It is well-known that if $(X_n)$ converges weakly to some distribution, then $X_n = O_p(1)$; and if to some distribution that does not charge 0 with positive probability, $X_n = \Omega_p(1)$. The sample mean of i.i.d.\ random variables with positive and finite variance is both $O_p(n^{-1/2})$ and $\Omega_p(n^{-1/2})$, but neither $O_{a.s.}(n^{-1/2})$ nor $\Omega_{a.s.}(n^{-1/2})$.

We use the symbols $\Pr(\cdot)$ and $\Exp(\cdot)$ to denote probability and expected value in a generic context; we use $P(\cdot)$ and $\Exp_P(\cdot)$ for probability and expected value when specifically $(X_n)\iid P$.

\paragraph{Additional related work.} A few other categories of recent research are surveyed in \cref{sec:related}, including (1) subGaussian and sub-$\psi$ test processes; (2) the contrasts among ``all'', ``$P$-almost all'', and ``$P$-most'' paths, and between ``always'' and ``eventually'' valid statements; (3) past papers that occasionally mention or hint at null bankruptcy.

\section{Bankruptcy of predictable plug-in and hedging}\label{sec:bankruptcy-predictable}
\subsection{Necessary and sufficient condition for null bankruptcy}

Our first theoretical contribution is that we identify the necessary and sufficient condition for the null bankruptcy \emph{event} (and consequently for this event to happen almost surely) of any predictable plugin betting strategy. As we cited earlier from \cite{waudby2024estimating}, all nonnegative martingales are some predictable plugin betting strategy, this completely characterizes the null bankruptcy behavior of all these processes.
The following result states that,
under any non-degenerate null $P$, bankruptcy happens exactly on the sample paths where the sum of squares of the bet fractions $\sum \lambda_n^2$ diverges, and sample paths where an ``all-in bet'' loses all current wealth, up to a $P$-negligible set.

\begin{theorem}[Sum-of-squares criterion]\label{thm:sos-crit} Let $P$ be a non-degenerate distribution on $[0,1]$ with mean $m$ and $(X_n)\iid P$.
    Let $\blamb = (\lambda_n)$ be a predictable process taking values in $[-\frac{1}{1-m}, \frac{1}{m}]$.
    Then, the $\blamb$-betting wealth process
    \begin{equation}
        W_n^{\blamb}=\prod_{k=1}^n (1 + \lambda_k(X_k - m))
    \end{equation}
    converges almost surely to a random variable $W_\infty$ satisfying
    \begin{equation}
       \{W_\infty = 0 \} = \left\{ \sum_{n=1}^\infty \lambda_n^2 = \infty \right\}  \cup \bigcup_{n=1}^\infty \{ \lambda_n(X_n - m) = -1 \} 
    \end{equation}
    and consequently 
    \begin{equation}
        \{ W_\infty > 0 \} = \left\{ \sum_{n=1}^\infty \lambda_n^2 < \infty \right\} \cap \bigcap_{n=1}^\infty \{ \lambda_n(X_n - m) > -1 \}.
    \end{equation}
\end{theorem}

We prove \cref{thm:sos-crit} in \cref{sec:pf-sos-crit}. The key steps include (1)
upper and lower bounding the log-wealth process $\log W_n^{\blamb}$ via
\begin{equation}
 x - x^2 \le \log(1+x) \le x
\end{equation}
for small $x$, reducing the problem to the analysis of the martingale $S_n= \sum_{k=1}^n \lambda_k(X_k - m)$ and its predictable quadratic variation;
and (2) martingale convergence and divergence characterizations found in \cite{hall2014martingale,Fitzsimmonslecturenotes}. 
It is worth noting that the event $\{ \lambda_n(X_n - m) = -1 \} $ only happens when both the bet fraction and the observation take extreme values:
\begin{equation}
     \{ \lambda_n(X_n - m) = -1 \}  = \left\{ \lambda_n = -\frac{1}{1-m}, X_n = 1 \right\} \cup \left\{ \lambda_n = \frac{1}{m}, X_n = 0 \right\}
\end{equation}
thus losing all current wealth.
We also note that \citet[Lemma 33]{ramdas2020admissible} also prove a sufficient condition for martingale bankruptcy, which, in this case, states that $P(\sum_{k=1}^\infty \lambda_k^2(X_k - m)^2 =\infty ) = 1$ implies $P(W_n^{\blamb} \to 0) = 1$. It is easy to see that our \cref{thm:sos-crit} implies their result, because $\sum_{k=1}^\infty \lambda_k^2(X_k - m)^2 =\infty$ implies $\sum_{k=1}^\infty \lambda_k^2 =\infty$, due to the boundedness $(X_k - m)^2 \le 1$.
Let us demonstrate in the following subsections that \cref{thm:sos-crit} facilitates the bankruptcy analysis of various betting strategies with explicit $(\lambda_n)$ expression.


\subsection{Almost sure divergence of $\sum_n \Omega_p(n^{-1})$}

Consider for a moment the simplest predictable plug-in betting strategy, the KT bettor \eqref{eqn:kt-lambda}.
From \cref{thm:sos-crit}, it is clear that KT \eqref{eqn:kt-lambda} with any $C$ is null-bankrupt if and only if the sum $\sum_n S_n^2/n^2$ diverges, where $S_n$ is the sum of $n$ i.i.d.\ mean-zero random variables $X_1 - m, \dots, X_n - m$. From the central limit theorem, we know that $S_n/n = \Omega_p(n^{-1/2})$, and therefore this sum is of the form $\sum_n \Omega_p(n^{-1})$. 

However, we were unable to locate existing work on the divergence of either the specific sum $\sum_n S_n^2/n^2$, or the general sums of form $\sum_n \Omega_p(n^{-1})$, even though these, at first sight, seem elementary problems. In particular, as we shall discuss soon, the law of the iterated logarithm (LIL) does \emph{not} imply $\sum_n S_n^2/n^2 = \infty$ almost surely. On the other hand, it is tempting to conjecture that, unlike $\sum_n \Omega_{a.s.}(n^{-1})=\infty$, the event $\sum_n \Omega_p(n^{-1}) = \infty$ does not necessarily happen almost surely, in light of various textbook counterexamples where convergence in probability does not imply convergence almost surely. Nevertheless, in the following theorem, we defy this conventional wisdom and assert that $\sum_n \Omega_p(n^{-1}) = \infty$ always happens almost surely.

\begin{theorem}\label{thm:sum-of-op}
    Let $(Z_n)$ be a nonnegative sequence of random variables such that $Z_n = \Omega_p(n^{-1})$. Then, $\Pr(\sum_{n=1}^\infty Z_n = \infty) = 1$.
\end{theorem}

The proof of \cref{thm:sum-of-op}, a result arguably of independent interest, is in \cref{sec:pf-sum-of-op}.
We can immediately apply the theorem to $Z_n = S_n^2/n^2 = \Omega_p(n^{-1})$ due to the central limit theorem. This special case, we believe, also deserves its separate attention and dissemination to the broader audience due to its simple form but not-that-simple proof. We therefore write it down separately.

\begin{corollary}\label[corollary]{lem:div-sample-mean}
    Let $Y_1,Y_2,\dots$ be i.i.d.\ random variables with mean 0 and variance $\sigma^2 > 0$ and $S_n = Y_1+\dots+Y_n$.
    Then,
    \begin{equation}
        \sum_{n=1}^\infty \frac{S_n^2}{n^2} = \infty \quad \text{almost surely}.
    \end{equation}
\end{corollary}

We point out that an invalid ``one-line'' proof attempt of \cref{lem:div-sample-mean} is the following: via the law of the iterated logarithm (LIL), $|S_n| \asymp \sqrt{n \log \log n}$ almost surely, and therefore $\sum_n S_n^2/n^2 \asymp \sum_n n^{-1} \log \log n = \infty$. The pitfall, we note, is that LIL only ensures that a \emph{subsequence} $(n_k)$ with $S_{n_k} = \Omega_{a.s.} (\sqrt{n_k \log\log n_k})$, and $\sum_{k} n_k^{-1} \log\log n_k$ may converge if  this subsequence $(n_k)$ is ``sparse'' (e.g.\ $n_k = k^2$). On the contrary, our statement is based conceptually on the fact that sample sizes $n$ such that $Z_n \asymp 1/n$ occupy a non-sparse subset among natural numbers: if one looks at the proof, the inclusions \eqref{eqn:incl1}, \eqref{eqn:incl2} state that $\sum_n Z_n$ converges only when, as $n$ grows, the number of events $A_k = \{ Z_k \ge \delta/k \}$ among $1\le k\le n$ happen grows sublinearly, which, we later show, is of low probability. This fact is not captured by LIL. Finally, an alternative proof of \cref{lem:div-sample-mean} based on Donsker's invariance principle is provided in \cref{sec:donsker}.

\subsection{Null bankruptcy of KT, GRAPA, aGRAPA, etc.} \label{sec:bankrupt-pred}

The combination of \cref{thm:sos-crit,thm:sum-of-op} leads to the following sufficient condition for bankruptcy. Namely, almost sure null bankruptcy happens if the bet fractions $(\lambda_n)$ have a decay rate $\Omega_p(n^{-1/2})$ or $\Omega_{a.s.}((n\log n)^{-1/2})$ under the null.

\begin{corollary}[$n^{-1/2}$ criterion] \label[corollary]{cor:sq-crit} Let $P$ be a non-degenerate distribution on $[0,1]$ with mean $m$ and $(X_n)\iid P$. Let $\blamb = (\lambda_n)$ be a predictable process taking values in $[-\frac{1}{1-m}, \frac{1}{m}]$.
    Suppose that either $\lambda_n = \Omega_p(n^{-1/2})$ or $\lambda_n = \Omega_{a.s.}((n\log n)^{-1/2})$ under $P$.
    Then, the $\blamb$-betting wealth process
    \begin{equation}
        W_n^{\blamb}=\prod_{k=1}^n (1 + \lambda_k(X_k - m))
    \end{equation}
    converges almost surely to 0.
\end{corollary}
Note that
$\lambda_n = \Omega_p(n^{-1/2})$ and $\lambda_n = \Omega_{a.s.}((n\log n)^{-1/2})$ are two very different conditions and neither implies the other.
Various predictable plug-in betting and hedging strategies mentioned in \cref{sec:intro} and by \cite{waudby2024estimating} fall into this category. First, the standard CLT for the sample mean immediately implies the following bankruptcy guarantee for KT.
\begin{proposition} Let $(X_n)\iid P$ with mean $m$ and variance $\sigma^2 > 0$. The KT bet fractions $\lKT_n$ defined in \eqref{eqn:kt-lambda} satisfies the asymptotic normality
    \begin{equation}\label{eqn:kt-lambda-clt}
   \sqrt{n} \lKT_n = \frac{1/2+\sum_{k=1}^{n-1}(X_k - m)}{C \sqrt{n}} \stackrel{\text{weakly}}{\longrightarrow} \mathcal{N}(0,C^{-2}\sigma^2).
\end{equation}
Therefore $\lKT_n = \Omega_p(n^{-1/2})$, and consequently the KT wealth process $\prod_{k=1}^n(1+\lKT_k(X_k - m))$ converges to 0 almost surely.
\end{proposition}

Second, GRAPA \eqref{eqn:grapa-lambda} is almost surely null-bankrupt because, as a standard M-estimator, the GRAPA bet fractions $(\lgr_n)$ also satisfy asymptotic normality and an $\Omega(n^{-1/2})$ decay rate under any non-degenerate null.
\begin{proposition}\label[proposition]{prop:grapa-bahadur} Let $(X_n)\iid P$ with mean $m$ and variance $\sigma^2 > 0$.
    The GRAPA bet fraction $\lgr_n$ defined in \eqref{eqn:grapa-lambda} satisfies the almost sure Bahadur expansion
    \begin{equation}\label{eqn:bahadur}
         \sqrt{n}\lgr_{n+1} = \frac{1}{\sigma^2} \cdot \frac{\sum_{k=1}^n (X_k - m)}{\sqrt{n}}  + o_{a.s.}(n^{-1/4} \log n) \stackrel{\text{weakly}}{\longrightarrow} \mathcal{N}(0,\sigma^{-2}).
    \end{equation}
    Therefore $\lgr_n = \Omega_p(n^{-1/2})$, and consequently the GRAPA wealth process $\prod_{k=1}^n(1+\lgr_k(X_k - m))$ converges to 0 almost surely.
\end{proposition}
The proof of \cref{prop:grapa-bahadur} involves checking that the M-estimation problem for $\lambda \mapsto \Exp_P(\log(1+\lambda(X_1 - m)))$ satisfies all of the seven regularity conditions for the Bahadur expansion result of M-estimators due to \cite{niemiro1992asymptotics}. We put these details in \cref{sec:grapa-bahadur}. As a quick sanity check of these two asymptotics, we note that KT with $C = m(1-m) = \sigma^2$ coincides with GRAPA in the Bernoulli coin-toss case (see e.g.\ \citet[Section 4]{orabona2016coin}).

Next, the approximate GRAPA (aGRAPA) bettor by \citet[Section B.3]{waudby2024estimating} reads
\begin{equation}
   \lagr_n = \left( -\frac{C}{1-m} \right) \vee \frac{\widehat \mu_{n-1} - m}{\widehat\sigma^2_{n-1}+(\widehat \mu_{n-1} - m)^2} \wedge \frac{C}{m} ,
\end{equation}
where $\widehat \mu_{n-1}$ and $\widehat\sigma^2_{n-1}$ are the sample mean and variance of $X_1,\dots, X_{n-1}$, and $C\in(0,1)$ a clipping constant (cf.\ the null Bahadur expansion of GRAPA \eqref{eqn:bahadur}). The consistency of these estimators as well as the standard CLT immediately give rise to its asymptotic normality similar to that of KT and GRAPA.
\begin{proposition} Let $(X_n)\iid P$ with mean $m$ and variance $\sigma^2 > 0$.
    The aGRAPA bet fraction $\lagr_n$ defined above satisfies the asymptotic normality
    \begin{equation}
         \sqrt{n}\lagr_{n}  \stackrel{\text{weakly}}{\longrightarrow} \mathcal{N}(0,\sigma^{-2}).
    \end{equation}
    Therefore $\lagr_n = \Omega_p(n^{-1/2})$, and the aGRAPA wealth process $\prod_{k=1}^n(1+\lagr_k(X_k - m))$ converges to 0 almost surely.
\end{proposition}

Finally,  predictable hedging betting strategies of form \eqref{eqn:hedged}, proposed on grounds of tightness of the implied confidence sequences (as opposed to wealth growth), are also almost surely null-bankrupt as $\lambda_n$ is set to always be $\Omega_{a.s.}(n^{-1/2})$ or $\Omega_{a.s.}((n \log n)^{-1/2})$, regardless of null or alternative.

\begin{proposition} Let $(X_n)\iid P$ with mean $m$ and variance $\sigma^2 > 0$.
    Let $\lambda_n^{\mathsf{PrH}}$ be the bet fraction defined in \eqref{eqn:hedging-rate}. 
    Then, $\lambda_n^{\mathsf{PrH}} = \Omega_{a.s.}((n \log n)^{-1/2})$, and consequently the hedged wealth process $0.5\cdot\prod_{k=1}^n(1+\lambda_k^{\mathsf{PrH}}(X_k - m)) +0.5\cdot \prod_{k=1}^n(1-\lambda_k^{\mathsf{PrH}}(X_k - m))$ converges to 0 almost surely.
\end{proposition}

As a short summary of this section, we have shown that many of the proposed ``good'' betting strategies are null-bankrupt almost surely via the sum-of-square criterion, \cref{thm:sos-crit}. It is natural to ask if there are ``equally good'' strategies that do not go bankrupt. We delay this profound question to \cref{sec:all}, after discussing mixture strategies among which the question has a much clearer answer.

\section{Bankruptcy of mixture strategies}

We next study the null bankruptcy behavior of mixture strategies \eqref{eqn:mixture-wealth}. While it is true that mixture strategies form a subclass of predictable plugin strategies via $\lambda_k^\pi = \frac{\int \lambda W_{k-1}^\lambda  \pi(\d \lambda) }{ \int W_{k-1}^\lambda  \pi(\d \lambda) }$, and therefore \cref{thm:sos-crit} implies the $\pi$-mixture strategy bankrupts if and only if $\sum_k (\lambda_k^\pi)^2$ diverges, we shall soon see that the bankruptcy of the mixture strategy is much easier to analyze directly via its native integration form
\begin{equation}
    W_n^\pi = \int W_n^\lambda \, \pi(\d \lambda) = \int \left\{\prod_{k=1}^n (1 + \lambda(X_k - m))  \right\} \pi(\d \lambda).
\end{equation}
Specifically, the only condition that determines if a mixture strategy is null-bankrupt is whether the mixture distribution $\pi$ has an atom at 0 (i.e.\ it charges the set $\{0\}$ with a positive probability $\pi(\{0\})>0$). Intuitively, if $\pi$ has an atom at 0, the mixture strategy always keeps some capital unwagered (``cash''), therefore it never goes bankrupt. The precise statement below is proved in \cref{sec:pf-nocash}.

\begin{theorem}[No-cash criterion]\label{thm:nocash} Let $P$ be a non-degenerate distribution on $[0,1]$ with mean $m$ and $(X_n)\iid P$.
Let $\pi$ be a probability measure  on $[-\frac{1}{1-m}, \frac{1}{m}]$. Then, the mixture wealth $(W_n^\pi)$ converges to $\pi(\{0\})$ almost surely. In particular, $W_\infty^\pi = 0$ almost surely if and only if $\pi$ does not have an atom at 0, i.e.\ $\pi(\{0\}) = 0$.
\end{theorem}
That is, any mixture betting strategy converges almost surely to the fraction of the mixture assigned to $\lambda = 0$. For any mixture distribution $\pi$ on $[-\frac{1}{1-m}, \frac{1}{m}]$, we can decompose the $\pi$-mixture strategy into its ``cash component'' $\pi|_{\{0\}}$ and its ``bet component'' $\pi|_{[-\frac{1}{1-m},0)\cup(0, \frac{1}{m}] }$, with the former staying constant and the latter going to bankruptcy.

In particular, the two mixtures employed by \cite{orabona2023tight} are both continuous, thus atomless at 0, and are consequently null-bankrupt.
\begin{proposition}\label[proposition]{prop:mixture-bankrupt}
    The universal portfolio betting strategy proposed by \citet[Section 4]{orabona2023tight}, which corresponds to $W_n^\pi$ where $\pi$ is a Beta distribution rescaled to $[-1,1]$, and the Robbins' iterated logarithm betting strategy proposed by \citet[Section 5]{orabona2023tight}, which corresponds to $W_n^\pi$ where $\pi$ has density $f_\pi(\lambda)=\frac{\id_{\{ |\lambda| \le 1 \} } \log \log C}{2|\lambda| \log(C/|\lambda|) (\log\log(C/|\lambda|))^2 }$ where $C=6.6e$, both converge to 0 almost surely under any non-degenerate null distribution.
\end{proposition}
Note that Robbins' mixture achieves the iterated logarithmic rate via being ``heavy at 0'': its \emph{density} satisfies $f_\pi(0) = \infty$. However, the mixture measure still charges 0 mass to the point $\{ 0 \}$, thus null-bankrupt.

We now revisit the question asked at the end of \cref{sec:bankruptcy-predictable}. Among mixture strategies, we know from \cref{thm:nocash} that null-bankrupt strategies are exactly those without the cash component $\pi(\{0\})$. Given any null-non-bankrupt mixture strategy, its bet component $\pi' = \pi|_{[-\frac{1}{1-m},0)\cup(0, \frac{1}{m}] }$ yields a strictly more powerful,  null-bankrupt strategy as
\begin{equation}\label{eqn:remove-cash}
    W^{\pi'}_n = \frac{W_n^\pi - \pi(\{ 0 \})}{1-\pi(\{ 0 \})} > W_n^\pi
\end{equation}
eventually on every sample path where $W^\pi_n \to \infty$. In this sense, all good \emph{mixture} betting strategies must be cash-free and go bankrupt almost surely under the null.

\section{Do all good strategies go bankrupt?}\label{sec:all}

We demonstrated above that known ``good'' mixture betting strategies are null-bankrupt. We now explore this principle in greater generality: do \emph{all ``good'' strategies} go bankrupt?

Revisiting our earlier argument around \eqref{eqn:remove-cash}, given an original strategy $M_n^\pi$, we constructed another strategy $M_n^{\pi'}$ that  makes more money under the alternative, at the price of making less money (and possible bankruptcy) under the null; it \emph{improves}
 upon the original strategy $M_n^\pi$ in this sense.

Our key result in this section is that it is possible to generalize the cash-removal improvement \eqref{eqn:remove-cash} to \emph{all} strategies
on certain sample paths which we refer to as being ``predictably non-bankrupt''. Let $(W_n)$ be the wealth process of some strategy whose bet fraction process is $(\lambda_n)$. Define
\begin{equation}
    \widehat W_n = \min_{x \in [0,1] } W_{n-1}(1 + \lambda_n(x - m)) = \min\{ W_{n-1}(1 + \lambda_n(0 - m)),  W_{n-1}(1 + \lambda_n(1 - m))  \}.
\end{equation}
That is, $\widehat W_n$ is the minimum possible wealth at time $n$ conditioned on the information available up to time $n-1$. Therefore, it forms a predictable process. For $\rho > 0$, we define the \emph{$\rho$-predictably non-bankrupt} event as
\begin{equation}\label{eqn:pnb}
    N^{\rho} := \bigcap_{n=1}^\infty \{ \widehat W_n > \rho \}.
\end{equation}
The event $N^\rho$ says that it is always guaranteed that the next-round wealth cannot drop below $\rho$. It implies wealth never \emph{actually} drops below $\rho$, and is implied by bet fraction being always small enough:
\begin{equation}
    \{ |\lambda_n| < 1-\rho W_{n-1}^{-1} \} \subseteq \{ \widehat W_n > \rho \} \subseteq \{ W_n > \rho \}.
\end{equation}
Clearly, for mixture strategies with cash component $\pi(\{ 0 \}) \ge \rho$ and non-degenerate bet component, the event $N^\rho$ is the entire space. Our result below states that we can improve any strategy on the event $N^\rho$.



\begin{theorem}[Improvability on predictably non-bankrupt paths] \label[theorem]{cor:improve}
    Let $(W_n)$ be the wealth process of some betting strategy, $\rho \in (0,1)$, and $N^\rho$ be its $\rho$-predictably non-bankrupt event as defined in \eqref{eqn:pnb}. There exists another betting strategy whose wealth process $(W_n^\sharp)$ satisfies $W_n^\sharp = \frac{W_n - \rho}{1 - \rho}$ on the event $N^\rho$. Consequently, on the event $N^\rho$:
    \begin{gather}
     \{ W_n > 1 \} \subseteq \{ W_n^\sharp > W_n  \}, \quad \{ W_n \to \infty \} \subseteq \{\liminf W_n^\sharp / W_n = 1/(1-\rho) > 1  \}; \label{eqn:improve-alt} \\
    \{ W_n < 1 \} \subseteq \{ W_n^\sharp < W_n  \}, \quad \{ W_\infty < 1 \} \subseteq \{ W_\infty^\sharp =  (W_\infty - \rho)/(1-\rho) < W_\infty \}. \label{eqn:improve-null}
    \end{gather}
\end{theorem}

To summarize \cref{cor:improve} in a nutshell: when the original strategy is predictably non-bankrupt, the improvement strategy makes more money under the alternative \eqref{eqn:improve-alt}, and loses more money under the null \eqref{eqn:improve-null}.
The intuition behind \cref{cor:improve} is the \emph{borrowing} or \emph{leveraging} 
nature of the cash-removal improvement \eqref{eqn:remove-cash} for mixture strategies: the cash-removed strategy $W_n^{\pi '}$ is equivalent to borrowing some cash and investing in a leveraged cash-holding strategy $W_n^\pi$. Analogously, the improvement strategy $W_n^{\sharp}$ in \cref{cor:improve} is equivalent to leveraging the original strategy $W_n$ on $N^\rho$. The full roadmap to developing these concepts of borrowing, the construction of $W_n^\sharp$, as well as the proof of \cref{cor:improve} can all be found in \cref{sec:pf-improve}.

As we admitted, the predictably non-bankrupt event $N^\rho$ is a subset of the actually non-bankrupt event $\bigcap_{n=1}^\infty \{ W_n > \rho \}$ (which must happen for some $\rho$ on a $W_\infty > 0$ path). Therefore, there is still a narrow gap between our result above and the general question ``do all good strategies go bankrupt''. We expect future work to close this gap.
We further our discussion on this intriguing question in \cref{sec:further-bankruptcy}, with examples and reasoning around (in fact, against) (1) whether exponentially powerful strategies are all null-bankrupt, (2) whether the Cram\'er-Rao bounds imply the necessary null-bankruptcy of some strategies, and (3) whether we can make the wealth lower bound $\rho$ a predictable sequence in \cref{cor:improve}.

\section{Odds and ends}

Our discussion on the null bankruptcy of testing-by-betting strategies naturally leads us to consider two adjacent topics that we have already hinted at. First, since these strategies are often compared against a hindsight optimum in regret analysis, it is natural to consider the null behavior of the hindsight optimum too. Second, bounded betting wealth processes are related to (and differ subtly from) subGuassian, or more generally, sub-$\psi$ testing processes, and we attempt to extend our techniques to the null bankruptcy analysis of these processes.

\subsection{Null asymptotics of $\klinf$}\label{sec:klinf}

Let $(W_n)$ be the wealth process of some betting strategy.
Much has been studied on the pathwise regret
\begin{equation}
    R_n = \max_{\lambda \in [-\frac{1}{1-m}, \frac{1}{m}]} (\log W_n^\lambda) - \log W_n.
\end{equation}
Having showed that $\log W_n \to -\infty$ on $P$-almost all paths for non-degenerate null $P$, and knowing that $R_n$ is usually $O(\log n)$ on \emph{all} sample paths from the online learning literature (e.g.\ KT \eqref{eqn:kt-lambda} in the binary case \citep{orabona2016coin}, the two mixtures mentioned in \cref{prop:mixture-bankrupt}),
we now investigate the behavior of
\begin{equation}
   L_n^* :=\max_{\lambda \in [-\frac{1}{1-m}, \frac{1}{m}]} (\log W_n^\lambda) = \max_{\lambda \in [-\frac{1}{1-m}, \frac{1}{m}]} \sum_{k=1}^n \log(1+\lambda(X_k - m)),
\end{equation}
the best-in-hindsight log-wealth, under the null.


It is worth noting that the quantity $L_n^*$ 
is better known as being related to the $\klinf$ statistic in the literature on multi-armed bandits. Many bandit methods \citep{thesis} are derived from controlling $\klinf(P_n, m)$
where $P_n$ is the empirical measure and $
    \klinf(P, m) = \min\{ \kl(P \| Q) : Q \text{ on }[0,1] \text{ with mean }m \}$.
Crucially, the $\klinf$ statistic defined via this minimization is shown by \cite{honda2010asymptotically} to have a dual representation that coincides with the hindsight maximum log-wealth, $n \klinf(P_n, m ) = L_n^*$.
We thus denote $\lKL_n = \argmax_{\lambda \in [-\frac{1}{1-m}, \frac{1}{m}]}  W_n^\lambda$, and so $L_n^* = W_n^{\lKL_n}$. Recalling the definition of GRAPA from \eqref{eqn:grapa-lambda}, we see that $\lgr_{n+1} = \lKL_n$.

An \emph{unconstrained} version of the hindsight maximum log-wealth, 
\begin{equation}
-\log \operatorname{EL}_n = \sup_{\lambda} \sum_{k=1}^n \log(1+\lambda(X_k - m))    ,
\end{equation}
where $\lambda$ can take any value as long as the logarithms are all defined, on the other hand, has been studied in the concept of the empirical likelihood by \cite{owen2001empirical} and the dual likelihood by \cite{mykland1995dual}. See also the discussion by \citet[Section E.6]{waudby2024estimating}. These authors show that the unconstrained supremum $-2\log \operatorname{EL}_n$ converges weakly to a $\chi^2_{(1)}$ limit. We note that as a well-behaved M-estimation procedure, adding the constraint $\lambda \in [-\frac{1}{1-m}, \frac{1}{m}]$ does not alter its asymptotic behavior, so the same $\chi^2_{(1)}$ limit applies to the constrained maximum $L_n^*$ as well. We prove this fact formally in \cref{sec:pf-chisq-klinf} using the Bahadur expansion of the GRAPA/$\klinf$ bet fractions $\lgr_{n+1} = \lKL_n$ in \cref{prop:grapa-bahadur}.

\begin{theorem}\label{thm:chisq-klinf}  Let $P$ be a non-degenerate distribution on $[0,1]$ with mean $m$ and $(X_n)\iid P$. Then, twice the hindsight maximum log-wealth
    \begin{equation}
  2  L^*_n =  2\max_{\lambda \in [-\frac{1}{1-m}, \frac{1}{m}]} \log W_n^\lambda = 2\sum_{k=1}^n \log(1 + \lKL_n (X_k - m) )
    \end{equation}
    converges weakly to a $\chi^2$ distribution with 1 degree of freedom. 
    Consequently, a null-bankrupt strategy must have unbounded regret on $P$-almost all paths:
    \begin{equation}\label{eqn:unbounded-regret}
        P(W_n \to 0)=1 \implies P( \sup(L_n^* - \log W_n)=\infty) = 1.
    \end{equation}
\end{theorem}

The hindsight maximum wealth $\prod_{k=1}^n (1 + \lKL_n (X_k - m) )$, therefore, converges weakly to $\exp(Z^2/2)$ where $Z \sim \mathcal{N}(0,1)$. We remark that this distribution is named the ``standard critical log-chi-squared distribution'' by \citet[Proposition 5.7]{Wang2025ensm}, and has infinite expected value.
This is in contrast to the almost sure bankruptcy of the GRAPA wealth $\prod_{k=1}^n (1 + \lKL_{k-1} (X_k - m) )$. These two have similar forms but have significantly different asymptotic behaviors, which is unsurprising: the hindsight maximum wealth is not a martingale, whereas the GRAPA wealth is. The unboundedness (i.e.\ $\omega(1)$ lower bound) of the regret \eqref{eqn:unbounded-regret} on almost all null paths complements the $O(\log n)$ all-path upper bounds in the literature.

Finally, we remark that both $\operatorname{EL}_n$ and $\klinf$ defined above are nonparametric likelihood ratio (NPLR) statistics:
\begin{gather}
    \operatorname{EL}_n = \sup \left\{ \frac{L(Q)}{L(P_n)} : \text{$Q$ with support on $\operatorname{supp}P_n$ and mean $m$} \right\},
    \\
 {-L_n^*}= {- n \klinf(P_n, m)} = \sup  \left\{ \log \frac{L(Q)}{L(P_n)} : \text{$Q$ with support on $[0,1]$ and mean $m$} \right\},
\end{gather}
where $L(\cdot)$ denotes the nonparametric likelihood, maximized globally by the empirical measure $P_n$. See e.g.\ \cite{gaffke2005three} for $\klinf$'s NPLR representation. Therefore, the asymptotic $\chi^2$ limit of both $-2\log \operatorname{EL}_n$ and $2L_n^*$ generalize the classical theorem of \cite{wilks1938large}.

\subsection{Null bankruptcy in subGaussian and sub-$\psi$ testing}\label{sec:crit-subg}

Counterparts of the divergence criteria \cref{thm:sos-crit,thm:nocash} can also be established for the plug-in and mixture strategies based on the subGaussian test martingale
\begin{equation}\label{eqn:subG-main}
    M^\lambda_n = \exp\left(\sum_{k=1}^n \frac{(X_k - m)^2 - (X_k - \lambda)^2}{2} \right),
\end{equation}
for which the online learning perspective is recently established by \cite{agrawal2025eventually}. See \cref{sec:related} for a short introduction.

\begin{theorem}[Sum-of-squares criterion II]\label{thm:sos-crit-subg}
 Let $P$ be a non-degenerate 1-subGaussian distribution with mean $m$ and $(X_n)\iid P$.
    Let $\blamb = (\lambda_n)$ be a predictable process taking values in $\mathbb R$.
    Then, the $\blamb$-plugin test process
    \begin{equation}
        M_n^{\blamb}= \exp\left(\sum_{k=1}^n \frac{(X_k - m)^2 - (X_k - \lambda_k)^2}{2} \right)
    \end{equation}
    converges almost surely to a random variable $M_\infty$ satisfying
    $$ \{M_\infty = 0 \} = \left\{ \sum_{n=1}^\infty (\lambda_n-m)^2 = \infty \right\},
    \text{\; and consequently \;} \{ M_\infty > 0 \} = \left\{ \sum_{n=1}^\infty (\lambda_n-m)^2 < \infty \right\}.$$
\end{theorem}

\begin{theorem}[No-cash criterion II]\label{thm:nocash-subg}  Let $P$ be a non-degenerate 1-subGaussian distribution with mean $m$ and $(X_n)\iid P$.  Let $\pi$ be a probability measure  on $\mathbb R$.
    Then, the $\pi$-mixture test process
    \begin{equation}
        M_n^{\pi}=\int \exp\left(\sum_{k=1}^n \frac{(X_k - m)^2 - (X_k - \lambda)^2}{2} \right) \pi(\d \lambda)
    \end{equation}
    converges almost surely to $\pi(\{m\})$.
\end{theorem}

Both theorems above are proved in \cref{sec:pf-subg}.
Finally, there is a generalization of the subGaussian mean testing martingale \eqref{eqn:subg-fixed} for the general \emph{sub-$\psi$} random variables. See e.g.\ \cite{howard2020time,howard2021time} for an introduction. In the sub-$\psi$ case, the process \eqref{eqn:subG-main} (taking $m=0$ for simplicity)
\begin{equation}
         M_n^{\blamb} = \exp\left\{ \sum_{k=1}^n \left(\lambda X_k - \frac{1}{2}\lambda^2 \right)  \right\} \quad \text{becomes} \quad
    \exp\left\{    \sum_{k=1}^n \left( \lambda X_k -\psi(\lambda)\right)\right\}
\end{equation}
where $\psi(\cdot)$ is a function that locally behaves like $\psi(x) \approx \frac{x^2}{2}$ for $x\approx 0$. Therefore, one may prove similar sum-of-squares and no cash criteria for these testing strategies. We omit these straightforward extensions from our paper.

\section{Conclusion}

Many successful betting strategies for the bounded mean testing problem converge almost surely to zero wealth under all non-degenerate null distributions, and we provided some preliminary insight that this principle may apply more broadly to all betting strategies that satisfy some growth condition.
We also discussed the null asymptotic $\chi^2$ distribution of the hindsight maximum wealth ($\klinf$); and presented the analogous bankruptcy results for the unbounded (sub-$\psi$) test martingales.
Our results are complementary to numerous results on the regret, null maximal concentration, and confidence sets corresponding to these strategies.

Our work also posed the natural question ``do all \emph{good} strategies go bankrupt?'' to which we answered by introducing predictably non-bankrupt paths and their improvability. Since predictable non-bankruptcy is stronger than bankruptcy, a gap remains between our answer and a fully satisfactory one to the question. We expect future work to close this gap.


\subsubsection*{Acknowledgments} AR was funded by NSF grant DMS-2310718.

\bibliography{BibTex}

\newpage
\appendix
\section{Additional related work}\label[appendix]{sec:related}

Our research on the null bankruptcy of bounded mean betting wealth processes is also related to the following topics.

\paragraph{SubGaussian and sub-$\psi$ mean testing.} Many of the betting strategies/wealth processes we mentioned in \cref{sec:intro} have counterparts in the (equally) classic Gaussian, subGaussian \citep{robbins1970statistical,robbins1968iterated}, and sub-$\psi$ \citep{howard2020time,howard2021time} mean testing literature.
Let $(X_n)\iid P$ where $P$ is a subGaussian distribution on $\mathbb R$ with variance factor 1 and unknown mean $\mu(P)$. That is, $\Exp_P \exp( t (X_1 - \mu(P)) ) \le \exp(t^2/2)$ for all $t \in \mathbb R$. The analogy of ``fixed fraction bet'' \eqref{eqn:bet-fixed} testing the null $\mu(P) = m$ is the likelihood ratio martingale
\begin{equation}\label{eqn:subg-fixed}
    M^\lambda_n = \prod_{k=1}^n \frac{p(X_k; \lambda, 1 )}{p(X_k; m, 1)} = \exp\left(\sum_{k=1}^n \frac{(X_k - m)^2 - (X_k - \lambda)^2}{2} \right),
\end{equation}
where $p(\cdot; \mu,\sigma^2 )$ is the probability density function of $\mathcal{N}(\mu,\sigma^2)$.
Predictable plugin 
\begin{equation}
    \exp\left(\sum_{k=1}^n \frac{(X_k - m)^2 - (X_k - \lambda_k)^2}{2} \right)
\end{equation}
or mixture
\begin{equation}
  \int  \exp\left(\sum_{k=1}^n \frac{(X_k - m)^2 - (X_k - \lambda)^2}{2} \right) \pi(\d \lambda)
\end{equation}
test processes that achieve universal power compared to the constant $\lambda$ are similarly available. More generally, we say $P$ is sub-$\psi$ with variance factor 1 if $\Exp_P \exp( t (X_1 - \mu(P)) ) \le \exp(\psi(t))$ for all $t \in [0, t_{\max})$, where $\psi$ is usually a function satisfying $\psi(t) \approx t^2/2$ when $t \to 0$ \citep{howard2020time}. The fixed-fraction test martingale testing the null $\mu(P) = m$ is now
\begin{equation}\label{eqn:subpsi-fixed}
    \exp\left\{    \sum_{k=1}^n \left( (\lambda - m) (X_k - m) -\psi(\lambda - m)\right)\right\}.
\end{equation}
When $\psi(t) = t^2/2$, \eqref{eqn:subpsi-fixed} equals \eqref{eqn:subg-fixed}. Predictable plug-in and mixture are similarly available.
In \cref{sec:crit-subg}, we discuss the null bankruptcy of these processes too.

\paragraph{All vs.\ $P$-almost all vs.\ $P$-most paths; always vs.\ eventually valid.} We have mentioned a line of work  either belonging to or inspired by the online learning and portfolio selection literature, notably that of \cite{cover1991universal,cover2002universal,orabona2016coin,orabona2023tight}. Indeed, the betting process \eqref{eqn:predictable} is equivalent to the online prediction game with logarithmic loss, with the total accumulated loss being $-\log W_n^{\blamb}$. These authors usually prove regret bounds that characterize the \emph{always-valid largeness} of $(W_n)$ on \emph{all} sample paths. The sequential inference literature \citep{howard2020time,waudby2024estimating}, on the other hand, draws heavily on the standard Ville's inequality which characterizes the \emph{always-valid smallness} of $(W_n)$ on \emph{$P$-most} sample paths: sample paths where $\sup(W_n) \le 1/\alpha$ is of $P$ measure at most $1-\alpha$.
Recently, \cite{agrawal2025eventually} deliver some very novel findings on the regret bounds that apply to $P$-most paths, not for the bounded betting game but for the \emph{subGaussian} game \eqref{eqn:subg-fixed}. They reason that in the subGaussian regime with unbounded observations, (always-valid) regret bounds may only apply to $P$-most paths.
We, on the other hand, primarily focus on the \emph{eventual smallness} of $(W_n)$ on \emph{$P$-almost all} sample paths. These sample paths constitute a larger set than a $P$-most set, but a proper subset of all paths. \cite{agrawal2025eventually} also notice that \emph{eventual largeness} statements (asymptotic regret bounds) can be proved on a $P$-almost all set of paths.

\paragraph{Other work on null-bankrupt test processes.} In the sequential statistics literature, many authors have proposed test processes that are nonnegative martingales, supermartingales, or e-processes under the null hypothesis, and mentioned along the way that these processes converge to 0 under (some) null distributions. These include \citet[Section 4.1]{ramdas2022testing} in the context of testing exchangeable bits, \citet[Table 3]{wang2025anytime} in the context of mixture-based t-tests and Z-tests. A sufficient condition for martingale bankruptcy is proposed by \citet[Lemma 33]{ramdas2020admissible}, 
which we discussed
when presenting a more useful necessary and sufficient condition \cref{thm:sos-crit}. Finally, it is noted by \cite{grunwald2023posterior} that test processes generalize the Bayesian prior-posterior ratio, and the null bankruptcy is therefore analogous to the concentration of posterior distribution towards the point mass on the ground truth, a concept visualized in passing in the bounded betting case by \citet[Figure 18]{waudby2024estimating}.

\section{Omitted and additional proofs}

\subsection{Proof of \cref{thm:sos-crit} (sum-of-squares criterion)}\label[appendix]{sec:pf-sos-crit}

\begin{proof}
    Consider the process $S_n= \sum_{k=1}^n \lambda_k(X_k - m)$, a square-integrable martingale with quadratic variation $ \langle S\rangle_n = \sigma^2 \sum_{k=1}^n \lambda_k^2$ where $\sigma^2 = \Var X_1 > 0$. 

    First, on the event $\{\sum_{n=1}^\infty \lambda_n^2 = \infty \} =  \{\langle S \rangle_\infty = \infty \}$, since $\Exp(\sup_n (S_n- S_{n-1})^2 ) \le \max( m^2/(1-m)^2, (1-m)^2/m^2 ) < \infty$, by \citet[Theorem 2(b)]{Fitzsimmonslecturenotes}, $S_n$ diverges almost surely. Further, on the event $\{S_n \text{ diverges} \}$, since $\Exp(\sup_n |S_n- S_{n-1}| ) < \infty$, we learn from \citet[Theorem 2.14]{hall2014martingale} that $\liminf S_n = -\infty$ almost surely. Because $W_n^{\blamb} \le \exp(S_n)$, this further implies that $\liminf W_n^{\blamb} = 0$ a.s.\ on this event. Since $W_n$ is a nonnegative martingale, it converges a.s.\ on the entire probability space to some $W_\infty$, so we can take $W_\infty$ such that $W_\infty = \lim W_n^{\blamb} = \liminf W_n^{\blamb} = 0$ on $\{\sum_{n=1}^\infty \lambda_n^2 = \infty \}$.

    Second, on the event $\left\{ \sum_{n=1}^\infty \lambda_n^2 < \infty \right\} \cap \bigcap_{n=1}^\infty \{ \lambda_n(X_n - m) > -1 \}$, there exists a random finite sample size $N$ such that when $n > N$, $|\lambda_n| \le 1/2$. Using the inequality $\log(1+x) \ge x - x^2 $ for $|x|\le 1/2$:
    \begin{equation}\label{eqn:three-parts-converge}
        \log W_n^{\blamb} \ge \sum_{k=1}^N \log( 1 + \lambda_k(X_k - m)) +\sum_{k= N+1}^n \lambda_k(X_k - \mu) - \sum_{k= N+1}^n \lambda_k^2(X_k - m)^2.
    \end{equation}
    First, the standard martingale convergence result from \citet[Theorem 2.15]{hall2014martingale} states that $S_n$ converges a.s.\ to a finite random variable on the event $\{ \sum_{n=1}^\infty \lambda_n^2 < \infty \} = \{ \langle S\rangle_\infty < \infty \}$; $\sum_{k=1}^n \lambda_k^2(X_k - m)^2$ also converges on this event as $| \lambda_k^2(X_k - m)^2| \le \lambda_k^2$. These indicate that the second and third terms in \eqref{eqn:three-parts-converge} both converge.
    The events $\{ \lambda_n(X_n - m) > -1 \}$ further ensure the first term in \eqref{eqn:three-parts-converge} is finite.
    Therefore, the right hand side of \eqref{eqn:three-parts-converge} converges a.s.\ to a finite variable on the event $\left\{ \sum_{n=1}^\infty \lambda_n^2 < \infty \right\} \cap \bigcap_{n=1}^\infty \{ \lambda_n(X_n - m) > -1 \}$. This implies we can ensure $W_\infty > 0$ on this event.

    Finally, on any event $\{ \lambda_n(X_n - m) = -1 \}$, it is clear that $W_k^{\blamb} = 0$ for all $k \ge n$.
\end{proof}

\subsection{Proof of \cref{thm:sum-of-op} (almost sure divergence of $\sum \Omega_p(n^{-1})$)}\label[appendix]{sec:pf-sum-of-op}

\begin{proof} Take any $\varepsilon \in (0,0.5)$.
    The condition $Z_n = \Omega_p(n^{-1})$ implies that there exist $\delta > 0$ and $N \ge 1$ such that $\Pr( n Z_n \ge \delta ) \ge 1-\varepsilon$ for all $n \ge N$.
    Consider the events $A_n = \{ nZ_n \ge \delta \}$ and the convergence event $C = \{ \sum_{n=1}^\infty Z_n < \infty \}$. We have,
    \begin{equation}
        Z_n \ge Z_n \id_{A_n} \ge \frac{\delta \id_{A_n}}{n}.
    \end{equation}
    Therefore,
    \begin{align}
        C &\subseteq \left\{ \sum_{n=1}^\infty \frac{\delta \id_{A_n}}{n} < \infty \right\} = \left\{ \sum_{n=N+1}^\infty \frac{\id_{A_n}}{n} < \infty \right\} \label{eqn:incl1} \\
        &\stackrel{(*)}\subseteq \left\{ \lim_{n \to \infty} \frac{1}{n}\sum_{k=N+1}^n \id_{A_k} = 0 \right\} =\left\{ \lim_{n \to \infty} \frac{1}{n-N}\sum_{k=N+1}^n \id_{A_k} = 0 \right\}. \label{eqn:incl2}
    \end{align}
    where the inclusion $(*)$ is due to Kronecker's lemma. Now, recalling that $\liminf E_n$ is the event that all but finitely many of events among $(E_n)$ happen simultaneously, and noting that $\lim a_n = 0$ implies that $a_n < 0.5$ for all but finitely many $n$,
    \begin{equation}
        \left\{ \lim_{n \to \infty} \frac{1}{n-N}\sum_{k=N+1}^n \id_{A_k} = 0 \right\} \subseteq \liminf_{n \to \infty} \left\{  \frac{1}{n-N} \sum_{k=N+1}^n \id_{A_k} < 0.5 \right\}.
    \end{equation}
    Next, noting that the random variable $G_n := \frac{1}{n-N} \sum_{k=N+1}^n \id_{A_k} \in [0,1]$ has expected value $\Exp G_n = \frac{1}{n-N}\sum_{k=N+1}^n \Pr(A_k) \ge 1-\varepsilon$, Markov's inequality implies
    \begin{equation}
        \Pr\left(  \frac{1}{n-N} \sum_{k=N+1}^n \id_{A_k} < 0.5 \right) = \Pr(1-G_n \ge 0.5) \le \frac{\Exp (1-G_n)}{0.5} \le 2\varepsilon.
    \end{equation}
    Combining everything above, we have
    \begin{equation}
        \Pr(C) \le \Pr \left(  \liminf_{n \to \infty} \left\{  \frac{1}{n-N} \sum_{k=N+1}^n \id_{A_k} < 0.5 \right\} \right) \stackrel{(**)}{\le} \liminf_{n\to\infty}  \Pr\left(  \frac{1}{n-N} \sum_{k=N+1}^n \id_{A_k} < 0.5 \right) \le 2\varepsilon
    \end{equation}
    where we applied Fatou's lemma to obtain the inequality $(**)$. Since $\varepsilon \in (0,0.5)$ is arbitrary, we see that $\Pr(C)=0$.
\end{proof}

\subsection{Divergence of $\sum S_n^2/n^2$ via Donsker's invariance principle}\label[appendix]{sec:donsker}

Below is an alternative proof of \cref{lem:div-sample-mean}.
\begin{proof} Assume $\sigma^2 = 1$ without loss of generality.
Let us study
\begin{equation}
    T_n = S_1^2 + \dots + S_n^2.
\end{equation}
By Donsker's invariance principle, as random variables in the Skorokhod space $\mathcal{D}[0,1]$, the functions
\begin{equation}
 B^{(n)}(t) :=  n^{-1/2} S_{\lfloor nt \rfloor}, \quad t \in[0,1]
\end{equation}
converge weakly to the standard Wiener process $B(t)$. Now consider the following function $F$ from the Skorokhod space $\mathcal{D}[0,1]$ to $\mathbb R$:
\begin{equation}
    F(f) = \int_0^1 f(t)^2 d t.
\end{equation}
We have
\begin{equation}
    F(B^{(n)}) = \int_0^1 n^{-1} S_{\lfloor nt \rfloor}^2 dt = n^{-1} \sum_{k=1}^n n^{-1} S_k^2 = n^{-2}T_n
\end{equation}
and
\begin{equation}
    F(B) =  \int_0^1 B^2(t) d t =: B^*.
\end{equation}
$F$ is a continuous function from the Skorokhod space to real numbers, a technical fact which we discuss later.
Therefore, the continuous mapping theorem implies that
\begin{equation}
    n^{-2} T_n \stackrel{\text{weakly}}{\longrightarrow} B^*,
\end{equation}
a random variable with a continuous distribution on $(0,\infty)$. For any $\delta > 0$, let $w_{\delta} > 0$ be the $\delta$-quantile of $B^*$. That is, $P( B^* \le w_\delta ) = \delta$. By Fatou's lemma,
\begin{equation}
    \Pr( \liminf\{ n^{-2} T_n \le w_\delta \} ) \le \liminf \Pr( n^{-2} T_n \le w_\delta   ) = \Pr( B^* \le w_\delta ) = \delta,
\end{equation}
where we recall that $\liminf A_n$ is the event that all but finitely many of events among $(A_n)$ happen simultaneously. 
Noting that $ \lim_{n \to \infty} n^{-2} T_n = 0$ would imply $ n^{-2} T_n \le w_\delta$ for all but finitely many $n$:
\begin{equation}
    \Pr( \lim n^{-2} T_n = 0 ) \le  \Pr(\liminf\{ n^{-2} T_n \le w_\delta \} ) \le \delta,
\end{equation}
and the arbitrariness of $\delta$ implies that $\Pr( \lim n^{-2} T_n = 0 ) = 0$.

Finally, by Kronecker's lemma, $\sum_{n=1}^\infty n^{-2} S_n^2 < \infty$ implies $n^{-2} T_n = n^{-2}\sum_{k=1}^n S_k^2 \to 0$. Since the latter happens with probability 0,
we see that $\sum_{n=1}^\infty n^{-2} S_n^2 = \infty$ almost surely. This completes the proof.
\end{proof}

Several remarks on this proof are as follows. First, the Skorokhod continuity of $F$ is proved below as \cref{lem:skrokhod-continuous}. Second, (yet) an alternative proof strategy that avoids the Skorokhod topology is to consider an appropriate ``linear interpolation'' of $B^{(n)}$, which converges to $B$ in the space of continuous functions (where the continuity of $F$ is straightforward).

\begin{lemma}\label[lemma]{lem:skrokhod-continuous}
    Let $\cD[0,1]$ be the space of c\`adl\`ag functions on $[0,1]$ equipped with the Skorokhod $J_1$ topology. Consider the square-integral functional $F:\cD[0,1]\to [0,\infty)$,
    \begin{equation}
        F(f) = \int_0^1 f^2(x) dx.
    \end{equation}
    Then, $F$ is continuous with respect to the standard topology on $[0,\infty)$ and the Skorokhod $J_1$ topology on $\cD[0,1]$.
\end{lemma}
\begin{proof}
    First, we observe that c\`adl\`ag functions on $[0,1]$ are all bounded, implying that $F(f) < \infty$ for all $f \in \cD[0,1]$.

    Next, given an $f \in \cD[0,1]$, assume $\|f(x)\|_\infty \le C$. Consider the Skorokhod $J_1$ metric $d$ defined as
    \begin{equation}
       d(f,g) := \inf_{\lambda \in \mathbb T} \left\{ \|f\circ \lambda - g\|_\infty + \|\lambda - id \|_\infty \right\}
    \end{equation}
    where $\mathbb T$ is the set of all increasing bijections on $[0,1]$. Consider the $\varepsilon$-ball around $f$:
    \begin{equation}
        B(f;\varepsilon) = \{ g \in \cD[0,1] : d(f,g)<\varepsilon \}.
    \end{equation}
    For any $g\in B(f;\varepsilon)$, there must exist a $\lambda\in \mathbb T$ such that
    \begin{equation}
         \|f\circ \lambda - g\|_\infty + \|\lambda - id \|_\infty \le \varepsilon.
    \end{equation}
    Therefore, for any $x \in [0,1]$,
    \begin{gather}
        g(x) \le f\circ\lambda(x) + \varepsilon \le \sup_{|y-x|\le\varepsilon} f(y) + \varepsilon,
        \\
        g(x) \ge f\circ\lambda(x) - \varepsilon \ge \inf_{|y-x|\le\varepsilon} f(y) - \varepsilon,
        \\
        |g(x)| \le C + \varepsilon.
    \end{gather}
    Therefore,
    \begin{align}
        |g^2(x) - f^2(x)| \le |g(x) + f(x)||g(x)-f(x)| \le (2C+\varepsilon)\left( \sup_{|y-x|\le \varepsilon}f(y) - \inf_{|y-x|\le \varepsilon}f(y) + 2\varepsilon\right).
    \end{align}
    Now, consider the two Riemann integrals
    \begin{equation}
        \int_0^1 \left(\sup_{|y-x|\le 1/2N}f(y) \right) dx \quad \text{and} \quad \int_0^1 \left(\inf_{|y-x|\le 1/2N}f(y)\right) dx.
    \end{equation}
    The Riemann sum of the first (resp.\ second) integral above on the partition
    \begin{equation}
        (0, 1/N, \dots, (N-1)/N, 1 ),
    \end{equation}
    which has mesh $1/N$,
    is close to the upper (resp.\ a lower) Darboux sum of $\int_0^1 f(x) dx$ on the partition
    \begin{equation}
        (0, 1/2N, 3/2N, \dots, (2N-1)/2N, 1),
    \end{equation}
    which has mesh $1/N$, and their difference (occurring only near the end points 0 and 1) is at most $2C/N$. Since $f$ is Riemann, hence Darboux, integrable,
    \begin{equation}
       \lim_{N \to \infty} \int_0^1 \left(\sup_{|y-x|\le 1/2N}f(y) - \inf_{|y-x|\le 1/2N}f(y) \right) dx = 0.
    \end{equation}
    It therefore follows that
    \begin{equation}
      \lim_{\varepsilon \to 0}  \sup_{ g\in B(f;\varepsilon) } |F(g) - F(f)| =  \lim_{\varepsilon \to 0} \int_0^1 (2C+\varepsilon)\left( \sup_{|y-x|\le \varepsilon}f(y) - \inf_{|y-x|\le \varepsilon}f(y) + 2\varepsilon\right) dx=0,
    \end{equation}
    concluding that $F$ is continuous at $f$.
\end{proof}

\subsection{Proof of \cref{prop:grapa-bahadur} (Bahadur expansion of GRAPA/$\klinf$ bet fractions)}\label[appendix]{sec:grapa-bahadur}

\begin{proof}
Define $f(\lambda, x) = -\log(1 + \lambda(x - m))$, $Q(\lambda) = \Exp_P{f(\lambda, X)}$, and
\begin{equation}
    Q_n(\lambda) = \frac{1}{n} \sum_{i=1}^n f(\lambda, X_i),
\end{equation}
where we consider the domain $\lambda \in [-1/(1-m), 1/m]$. Let $\lambda^*$ and $\lKL_n$ be the minimizers of $Q$ and $Q_n$ respectively. Then, $\lgr_{n+1} = \lKT_n$. (See \cref{sec:klinf} for the concept of $\lKT_n$ and $\klinf$.)

We now verify that the M-estimation problem above meets all assumptions for Theorem 5 in \cite{niemiro1992asymptotics}, with regularity constants $s=0$ and $r=10$. Note that while the asymptotic results of \cite{niemiro1992asymptotics} are stated for open domains, they also apply to our closed domain $ [-1/(1-m), 1/m]$ as the minimizer is allowed to be defined arbitrarily when $Q_n$ has no minimum on the open domain $(-1/(1-m), 1/m)$.

\newcommand{\E}[1]{\Exp_P\left(#1\right)}
\begin{itemize}
    \item $f(\cdot,x)$ is convex for any $x\in[0,1]$.
    \item For any $\lambda \in (-1/(1-m), 1/m$, the expected value $Q(\lambda) = \E{f(\lambda, X)}$ is finite, since $f(\lambda, \cdot)$ is an upper and lower bounded function.
    \item $\lambda^* = 0$ is the unique minimizer of $Q(\lambda)$. To see this, $Q(0) = 0$ and if $\lambda \neq 0$, because (1)
    $t \mapsto -\log(1+t)$ is strictly convex, (2) $\lambda(X-\mu)$ is not a point mass:
\begin{equation}
     Q(\lambda) = \E{  -\log(1 + \lambda(X - m)) } > -\log(1 + \E{\lambda(X - \mu)}) = 0.
\end{equation}
\end{itemize}
The properties above already ensured $\lKT_n \to \lambda^* $ almost surely. The following additional properties ensure normality and almost sure Bahadur expansion. Let $g(\lambda, x) =f'(\lambda, x) = \frac{-x + m}{1 + \lambda(x-m)}$.
\begin{itemize}
    \item $\E{ g^2(\lambda, X) } < \infty$ for each $\lambda$, since $g(\lambda, \cdot)$ is bounded.
    \item $Q(\lambda)$ is twice differentiable at the minimum $\lambda = 0$, and $Q''(0) > 0$. To see this, for small $\lambda$:
    \begin{equation}
        Q''(\lambda) = \E{ g'(\lambda, X) } = \E{ \frac{(X - m)^2}{ (1 + \lambda(X-m))^2 }  }, \quad Q''(0) = \Var{X} > 0.
    \end{equation}
    \item Expanding $Q'(\lambda) = \E{g(\lambda, X)}$ near $\lambda \approx 0$:
    \begin{align}
        & Q'(\lambda) = \E{ g(\lambda, X)  } = \E{ g(0, X) + \lambda g'(0, X)  + \frac{\lambda^2}{2} g''(\xi_X , X) } =
        \\
        &\lambda \Var X - \frac{\lambda^2}{2} \E{ \frac{2(X - m)^3}{ (1 + \xi_X(X-m))^3 } }
    \end{align}
    where $|\xi_X| \le |\lambda|$ is the Lagrange remainder that depends on $X$. It is easy to see that if $|\lambda| < 1/2$,
    \begin{equation}
         |Q'(\lambda) - \lambda Q''(0)| \le 8 \lambda^2.
    \end{equation}
    \item Consider $\E{ (g(\lambda , X) - g(0 ,X) )^2 }$ near $\lambda \approx 0$. With a similar Lagrange remainder $|\zeta_X| \le |\lambda|$:
    \begin{align}
        &\E{ (g(\lambda , X) - g(0 ,X) )^2 } = \E{ (X-m)^2 \left( 1 - \frac{1}{1 +\lambda(X - m)} \right) ^2} =
        \\
        &\E{ (X-m)^2 \left( \frac{\lambda (X - \mu) }{  (1 + \zeta_X(X - m))^2   }   \right) ^2} .
    \end{align}
    Therefore if $|\lambda| < 1/2$,
    \begin{equation}
         \E{ (g(\lambda , X) - g(0 ,X) )^2 }  \le  16 \lambda^2.
    \end{equation}
    \item Finally, for $|\lambda| < 1/2$,
    \begin{equation}
        \E{ g^{10}(\lambda, X) } = \E{  \frac{|X - m|^{10}}{|1 + \lambda(X-m)|^{10}} } \le 1024.
    \end{equation}
\end{itemize}
From these properties, we have verified that all conditions for Theorem 5 in \cite{niemiro1992asymptotics} hold with $s=0$ and $r=10$. It thus follows that, \emph{almost surely},
\begin{align}
    \sqrt{n}\lKL_n &=  - \frac{1}{Q''(0)} \cdot \frac{\sum_{k=1}^n g(0, X_k)}{\sqrt{n}} + o(n^{-1/4} \log n) \\
    &= \frac{1}{\Var X} \cdot \frac{\sum_{k=1}^n (X_k - m)}{\sqrt{n}}  + o(n^{-1/4} \log n),
\end{align}
which concludes the proof.
\end{proof}

We also remark that the Bahadur expansion of $\lgr_{n+1} = \lKL_n$ above can be combined with \cref{lem:div-sample-mean} to show that $\sum (\lgr_{n})^2 = \infty$ almost surely, thus an alternative proof of the bankruptcy of GRAPA.

\subsection{Proof of \cref{thm:nocash} (no-cash criterion)}\label[appendix]{sec:pf-nocash}

The proof of \cref{thm:nocash} is made convenient by the following lemma.

\begin{lemma}\label[lemma]{lem:exp-regular}
    Let $\gamma$ be a measure on $[0,\infty)$. Then,
    \begin{equation}
        \lim_{n \to \infty} \int_0^\infty \exp( - x A_n ) \gamma(\d x) = \gamma(\{0\})
    \end{equation}
    for any positive increasing sequence $( A_n )$ that increases to $\infty$.
\end{lemma}
\begin{proof}
    The functions $( \exp(-xA_n) : x \in [0,\infty) )_{n \ge 1}$ are pointwise monotone decreasing, and bounded in $[0,1]$. Therefore, due to the monotone convergence theorem,
    \begin{equation}
                \lim_{n \to \infty} \int_0^\infty \exp( - x A_n ) \gamma(\d x) =   \int_0^\infty \left\{ \lim_{n \to \infty}  \exp( - x A_n ) \right\} \gamma(\d x) =  \int_0^\infty \id_{\{ x=0 \}} \gamma(\d x) =\gamma(\{0\}),
    \end{equation}
    concluding the proof.
\end{proof}

Now we prove \cref{thm:nocash}.

\begin{proof}
First, let us prove the case when $\pi$ does not have an atom at 0, i.e., $\pi(\{0\}) = 0$.
Let $\pi^+$ and $\pi^-$ be the restrictions of $\pi$ on $(0,1/m]$ and $[-1/(1-m),0 )$ respectively. Since $\pi$ does not have an atom at 0,
    \begin{equation}
    W_n^\pi = \underbrace{\int_0^{1/m} \left\{\prod_{k=1}^n (1 + \lambda(X_k - m))  \right\} \pi^+(\d \lambda)}_{W^+_n} +  \underbrace{\int_{-1/(1-m)}^0 \left\{\prod_{k=1}^n (1 + \lambda(X_k - m))  \right\} \pi^-(\d \lambda)}_{W^-_n}.
\end{equation}
$(W^{+}_n)$ and $(W^-_n)$ are both nonnegative martingales (since they are mixtures of nonnegative martingales), and let us show that they both converge to 0 almost surely. By the inequality $1+x \le \exp(x)$:
\begin{equation}
    W_n^+ \le \int_0^{1/m} \exp\left( \lambda \sum_{k=1}^n(X_k - m)  \right) \pi^+(\d \lambda).
\end{equation}
By the law of the iterated logarithm, $\liminf \sum_{k=1}^n (X_k - m) = -\infty$ almost surely, so
there exists a (random) sequence of positive integers $n_1 < n_2 < \dots$ such that $s_N := -\sum_{k=1}^{n_N}(X_k - m) $ is a positive increasing sequence that increases to $\infty$. By \cref{lem:exp-regular}, the sequence $(W_{n_N}^+)_{N \ge 0}$ then converges to 0. Therefore we see that,
\begin{equation}
    \liminf_{n \to \infty}  W_n^+ = 0, \quad\text{almost surely}.
\end{equation}
However, $W_n^+$ is a nonnegative martingale, so it must converge almost surely, therefore
\begin{equation}
    \lim_{n \to \infty}  W_n^+ = 0, \quad\text{almost surely}.
\end{equation}
The same reasoning holds for $W_n^-$, using $\limsup \sum_{k=1}^n (X_k - m) = \infty$ almost surely. This concludes the proof that $W_n^\pi \to 0$ almost surely.

Next, if $\pi$ has an atom at 0, i.e., $\pi(\{0\}) > 0$, we can simply decompose the wealth into its ``cash component'' and its ``bet component''
\begin{equation}
    W_n^\pi = \pi(\{0\}) \cdot W_n^0 + \int W_n^\lambda \, \pi|_{[-\frac{1}{1-m},0)\cup(0,\frac{1}{m}]}(\d \lambda),
\end{equation}
where $W_n^0 = 1$, and the second part converges to 0. Therefore, $W_n^\pi \to \pi(\{0\})$ almost surely.
\end{proof}

\subsection{Discussion and proof for \cref{cor:improve} (improvability of non-bankrupt strategies)}\label[appendix]{sec:pf-improve}

We begin our argument by noting that the cash-removal \eqref{eqn:remove-cash} is actually a mixture of two strategies, $\pi$-mixture and cash ($M_n^{\bf 0} = 1$), with a \emph{negative} weight on cash. That is, one leverages (longs) the original strategy $\pi$-mixture by borrowing (shorting) some cash.
Let us therefore clarify the general scenarios for the mixture of two strategies that may or may not allow shorting.

\begin{fact}[Mixture of two predictable plug-ins, long only]\label[fact]{fct:mixture-of-two} Let $(W_n^{\blamb})$ and $(W_n^{\bnu})$ be the wealth processes corresponding to two predictable bet fraction sequences $\blamb$ and $\bnu$. Then, the portfolio consisting of these two strategies
\begin{equation}
   (1-\kappa) W_n^{\blamb} + \kappa W_n^{\bnu}
\end{equation}
is equivalent to the predictable plug-in strategy $W_n^{\bbeta}$ with bet fraction sequence
\begin{equation}\label{eqn:fraction-mixture-1}
    \beta_n = \frac{(1-\kappa)W_{n-1}^{\blamb} \lambda_n +  \kappa W_{n-1}^{\bnu} \nu_n }{(1-\kappa)W_{n-1}^{\blamb} + \kappa W_{n-1}^{\bnu}},
\end{equation}
where $\kappa \in [0,1]$ is a constant.
\end{fact}

To see that the bet fraction \eqref{eqn:fraction-mixture-1} $\beta_n \in [-\frac{1}{1-m}, \frac{1}{m}]$, it is a convex combination of $\lambda_n, \nu_n \in [-\frac{1}{1-m}, \frac{1}{m}]$.

\begin{fact}[Mixture of two predictable plug-ins, short allowed]\label[fact]{fct:short} Let $(W_n^{\blamb})$ and $(W_n^{\bnu})$ be the wealth processes corresponding to two predictable bet fraction sequences $\blamb$ and $\bnu$. If these strategies further ensure that, \textbf{on all binary paths $(X_n) \in \{0, 1\}^{\infty}$},
\begin{equation}\label{eqn:bound-on-bin-paths}
    \frac{W_n^{\blamb}}{W_n^{\bnu}} > \rho \in [0,1), \quad \forall n
\end{equation}
then, the portfolio consisting of these two strategies
\begin{equation}
   (1-\kappa) W_n^{\blamb} + \kappa W_n^{\bnu}
\end{equation}
is equivalent to the predictable plug-in strategy $W_n^{\bbeta}$ with bet fraction sequence
\begin{equation}\label{eqn:fraction-mixture-2}
    \beta_n = \frac{(1-\kappa)W_{n-1}^{\blamb} \lambda_n +  \kappa W_{n-1}^{\bnu} \nu_n }{(1-\kappa)W_{n-1}^{\blamb} + \kappa W_{n-1}^{\bnu}},
\end{equation}
where we can now \textbf{long $\blamb$ and short $\bnu$}:
\begin{equation}
    \kappa \in \left[ -\frac{\rho}{1 - \rho}, 1 \right].
\end{equation}
In particular, when $\bnu = \bf 0$, it means that starting with any strategy $\blamb$ with always-valid strict wealth lower bound $W_n^{\blamb} > \rho$ on all binary paths, one can borrow $b = -\kappa \in (0, \rho/(1-\rho)]$ amount of cash and execute the leveraged bets
\begin{equation}
    \beta_n =  \frac{(1+b)W_{n-1}^{\blamb} \lambda_n  }{(1+b)W_{n-1}^{\blamb} - b}\quad \implies \quad W_n^{\bbeta} = (1+b)W_n^{\blamb} - b
\end{equation}
instead. 
\end{fact}
To see that the bet fraction $\beta_n$ defined in \eqref{eqn:fraction-mixture-2} is in $[-\frac{1}{1-m}, \frac{1}{m}]$ even when $\kappa < 0$,
$$1+\beta_n(X_n - m) = \frac{  (1-\kappa) W_n^{\blamb} + \kappa W_n^{\bnu}}{  (1-\kappa) W_{n-1}^{\blamb} + \kappa W_{n-1}^{\bnu}} > 0$$
for all binary paths, it holds in particular when $X_n = 0$ and $1$. This is why we require the bound \eqref{eqn:bound-on-bin-paths} to hold on binary paths when formulating shorting here: if we merely know by oracle that $ \frac{W_n^{\blamb}}{W_n^{\bnu}} > \rho $ without knowing $X_n$ can take $0$ and $1$, the bet fraction equivalent to shorting \eqref{eqn:fraction-mixture-2} might correspond to an infeasible fraction that just ``happens not to'' lead to negative wealth. The assumption of \eqref{eqn:bound-on-bin-paths} on all binary paths is weaker than \eqref{eqn:bound-on-bin-paths} on all bounded paths $(X_n)\in [0,1]^\infty$, and in general incomparable to \eqref{eqn:bound-on-bin-paths} on $P$-almost all paths under some distribution $P$.

We also see here that for a mixture strategy with cash component $\pi(\{0\}) > 0$, removing the cash component is an instance of the above with $\rho = \pi(\{0\})$ and $b =  \rho/(1-\rho)$. However, in general, a non-bankrupt strategy's minimum wealth is a path-dependent quantity that is not known in advance. We thus employ the following idea: we fix \emph{some} $\rho > 0$, and execute the bet fraction equivalent to borrowing $b=\rho/(1-\rho)$
\begin{equation}
    \beta_n^{\rho} = \frac{ (1+\frac{\rho}{1-\rho}) W_{n-1}^{\blamb} \lambda_n }{(1+\frac{\rho}{1-\rho}) W_{n-1}^{\blamb} - \frac{\rho}{1-\rho}} = \frac{W_{n-1}^{\blamb} \lambda_n}{W_{n-1}^{\blamb} - \rho}
\end{equation}
when it is valid (i.e.\ $\beta_n^\rho \in [-\frac{1}{1-m}, \frac{1}{m}]$). Analogous to the discussion around the validity of \eqref{eqn:fraction-mixture-2} above, we have the following statement characterizing when the validity $\beta_n^\rho \in [-\frac{1}{1-m}, \frac{1}{m}]$ holds. We recall from \cref{sec:all} that
\begin{equation}
    \widehat W_n^{\blamb} = \min_{x \in [0,1] } W_{n-1}^{\blamb}(1 + \lambda_n(x - m)) = \min\{ W_{n-1}^{\blamb}(1 + \lambda_n(0 - m)),  W_{n-1}^{\blamb}(1 + \lambda_n(1 - m))  \}.
\end{equation}
\begin{lemma}
    $\{ \widehat W_n^{\blamb} > \rho \} \subseteq \{ -\frac{1}{1-m} < \beta_n^\rho < \frac{1}{m} \}$.
\end{lemma}
\begin{proof} On the event $\{ \widehat W_n^{\blamb} > \rho \}$, 
    we have $W_{n-1}^{\blamb}(1+ \lambda_n(x-m) ) > \rho$ for $x=0$, $m$, and $1$. So
    \begin{equation}
        1 + \beta_n^\rho(x - m) = \frac{\frac{1}{1-\rho}W_{n-1}^{\blamb}(1+ \lambda_n(x-m) ) - \frac{\rho}{1-\rho}  }{\frac{1}{1-\rho}W_{n-1}^{\blamb} - \frac{\rho}{1-\rho}} > 0, \quad x \in \{0, 1\}.
    \end{equation}
    Therefore, $\beta_n^\rho < \frac{1}{m}$ by $x = 0$, and $\beta_n^\rho > -\frac{1}{1-m}$ by $x=1$.
\end{proof}


That is, one can bet $\beta_n^\rho $ when $\{ \widehat W_n^{\blamb} > \rho \}$ happens. This leads us to the following definition of a new strategy that ``opportunistically'' executes the $\rho$-leverage.

\begin{definition}[Opportunistic leveraging] Let $\blamb = (\lambda_n)$ be a predictable plugin betting strategy and $\rho > 0$. Define $E_n = \{ \widehat W_n^{\blamb} > \rho \}$. Then the events $(E_n)$ are predictable. Define the new plugin fraction
\begin{equation}
    \gamma_n = \id_{ E_n } \frac{W_{n-1}^{\blamb} \lambda_n}{W_{n-1}^{\blamb} - \rho} + (1-\id_{ E_n }) \lambda_n,
\end{equation}
which is predictable and always belongs to $[-\frac{1}{1-m}, \frac{1}{m}]$. Therefore $\bgam = (\gamma_n)$ is also a predictable plugin betting strategy.
Further, on $\cap_{k=1}^n E_k$, $W_n^{\bgam} = \frac{W_n^{\blamb} - \rho}{1 - \rho}$. 
We call this new strategy $\bgam$ the \emph{$\rho$-opportunistic leverage} of the strategy $\blamb$.
\end{definition}

In words, the transformation above defines a new strategy that, at each step, first checks if it is safe to execute a bet fraction corresponding to borrowing $b = \rho/(1-\rho)$ by seeing if the precondition $E_n = \{ \widehat W_n^{\blamb} > \rho \}$ holds, then does so if it is and executes the original bet fraction if otherwise. Thus, if next-step minimum wealth $\widehat W_n^{\blamb} $ is indeed always above $\rho$, the new strategy's wealth process is indistinguishable from borrowing $b = \rho/(1-\rho)$ and leveraging the original strategy from the outset. This transformation leads to the proof of \cref{cor:improve}.

\begin{proof}
   Let $(W_n^\sharp)$ be the wealth process of the {$\rho$-opportunistic leverage} of the original strategy. On the event $N^\rho = \bigcap_{n=1}^\infty E_n$, $W_n^\sharp = \frac{W_n - \rho}{1-\rho}$ for all $n$, thus $W_n^\sharp < W_n$ if $W_n < 1$, $W_n^\sharp > W_n$ if $W_n > 1$, and $W_n^\sharp / W_n \to 1/(1-\rho)$ if $W_n \to \infty$.
\end{proof}

\begin{remark}
    The construction above uses the concept of borrowing. Recently, \cite{wang2024borrow} discuss at length the definition, effect, and utility of borrowing in testing-by-betting. Our definition in \cref{fct:short} corresponds to the ``net wealth'' (i.e.\ always paying back the debt after borrowing) discussed by \cite{wang2024borrow}, and differs from it in that we do not allow negative wealth.
\end{remark}

\subsection{Proof of \cref{thm:chisq-klinf} ($\chi^2_{(1)}$ limit of $\klinf$)}\label[appendix]{sec:pf-chisq-klinf}

\begin{proof}
    Due to the Bahadur expansion of the GRAPA bet fractions in \cref{prop:grapa-bahadur},
    \begin{equation}
       \lKL_{n} = \frac{1}{\sigma^2} \cdot \frac{\sum_{k=1}^n (X_k - m)}{n}  + r_n 
    \end{equation}
    where $r_n = o_{a.s.}(n^{-3/4} \log n)$. Expanding $\log(1+x)$, we have
    \begin{align}
          L^*_n =&  \sum_{k=1}^n \log(1 + \lKL_n (X_k - m) ) = \sum_{k=1}^n \lKL_n (X_k - m)  -  \sum_{k=1}^n \frac{(\lKL_n)^2 (X_k - m)^2}{2(1+\xi_{nk})^2}, \\
           =& \frac{S_n^2}{n\sigma^2} + r_n S_n - \frac{( - S_n + n\sigma^2 r_n )^2}{2n \sigma^2}  \sum_{k=1}^n \frac{(X_k - m)^2}{n\sigma^2(1+\xi_{nk})^2} = \frac{S_n^2}{2n\sigma^2} + o_{a.s.}(n^{-1/4}\log^2 n).
    \end{align}
    where $|\xi_{nk}| \le |\lKL_n| = o_{a.s.}(\sqrt{n^{-1} \log n})$ is the Lagrange remainder, and $S_n = \sum_{k=1}^n (X_k - m) = o_{a.s.}(\sqrt{n\log n})$. Noting that $\frac{S_n^2}{n\sigma^2}$ converges weakly to $\chi^2_{(1)}$, we see that so does $2L_n^*$ via Slutsky's theorem.
\end{proof}

\subsection{Proofs of the subGaussian criteria, \cref{thm:sos-crit-subg,thm:nocash-subg}}\label[appendix]{sec:pf-subg}

The proof of \cref{thm:sos-crit-subg} is below.
\begin{proof}
    Without loss of generality, let $m=0$. Consider the log-process
    \begin{equation}\label{eqn:log-subg}
        \log M_n^{\blamb} = \sum_{k=1}^n \lambda_k X_k - \frac{1}{2} \sum_{k=1}^n \lambda_k^2.
    \end{equation}
    On the event $\{ \sum \lambda_n^2 < \infty \}$, the martingale $\sum_{k=1}^n \lambda_k X_k$ has convergent quadratic variation, therefore it a.s.\ converges due to \citet[Theorem 2.15]{hall2014martingale}, so $\log  M_n^{\blamb}$ converges a.s.\ on this event. Therefore, $M_n^{\blamb}$ converges a.s.\ to a non-zero random variable on this event.

    On the event $\{ \sum \lambda_n^2 = \infty \}$, we invoke the martingale strong law of large numbers \citep[Theorem 3]{Fitzsimmonslecturenotes} and see that $\frac{\sum_{k=1}^n \lambda_k X_k}{\sum_{k=1}^n \lambda_k^2} \to 0$ almost surely. Therefore, on this event,
    \begin{equation}
          \log M_n^{\blamb} = \left(\frac{\sum_{k=1}^n \lambda_k X_k}{\sum_{k=1}^n \lambda_k^2} - \frac{1}{2}\right)\cdot\left({\sum_{k=1}^n \lambda_k^2}\right) \to (0-0.5)\cdot(+\infty) = -\infty
    \end{equation}
    almost surely. So $M_n^{\blamb} \to 0$ a.s.\ on this event.
\end{proof}
The proof of \cref{thm:nocash-subg} is below.
\begin{proof}
    It suffices to consider $m = 0$ and $\pi(\{ 0 \}) = 0$. Similar to the proof of \cref{thm:nocash}, we show that the nonnegative martingales
    \begin{equation}
        M_n^+ = \int_0^\infty\exp\left(\sum_{k=1}^n \frac{X_k^2 - (X_k - \lambda)^2}{2} \right) \pi^+(\d \lambda), \;  M_n^- = \int_{-\infty}^0\exp\left(\sum_{k=1}^n \frac{X_k^2 - (X_k - \lambda)^2}{2} \right) \pi^-(\d \lambda)
    \end{equation}
    both converge a.s.\ to 0, where $\pi^+$ and $\pi^-$ are the $(0,\infty)$ and $(-\infty,0)$ parts of $\pi$ respectively. Note that
    \begin{equation}
        M_n^+ = \int_0^\infty\exp\left(\sum_{k=1}^n \frac{2 \lambda X_k - \lambda^2}{2} \right) \pi^+(\d \lambda) \le \int_0^\infty\exp\left(\lambda \sum_{k=1}^n X_k \right) \pi^+(\d \lambda).
    \end{equation}
    Using the same technique as in the proof of \cref{thm:nocash}, $\liminf \sum_{k=1}^n X_k = -\infty$ a.s.\ and \cref{lem:exp-regular} imply that $\lim M_n^+ = \liminf M_n^+ = 0$ a.s.. The same works for $M_n^-$.
\end{proof}

\section{Further discussions on good strategies' necessary bankruptcy} \label[appendix]{sec:further-bankruptcy}

In this appendix, we further discuss the topic in \cref{sec:all}.
We first present additional examples exploring the question ``do all good strategies go bankrupt''.
Below, a slightly modified predictable plug-in strategy is no longer null-bankrupt,  still remains universally power, but is always only \emph{sub}exponentially powerful.

\begin{example}[Increasingly intermittent bets]\label[example]{ex:intermittent} Let $\alpha > 1$ and consider the set $I_\alpha = \{ \lfloor n^\alpha \rfloor : n = 1,2,\dots \}$.
    Let $\lambda_n' = \id_{ n \in I_{\alpha} } \lambda_n$, where $\lambda_n$ is either the KT bet fraction $\lKT_n$ with $C=2$, the GRAPA bet fraction $\lgr_n$, or the aGRAPA bet fraction $\lagr_n$. Then, the wealth process
    \begin{equation}
        W_n' = \prod_{k=1}^n (1 + \lambda_k'(X_k -m) = \prod_{k\in I_{\alpha}, k\le n } (1 + \lambda_k'(X_k-m))
    \end{equation}
    grows sub-exponentially fast $\log W_n' =  \Theta_{a.s.}(n^{1/\alpha})$ under any alternative distribution, and converges a.s.\ to a positive random variable under any null distribution.
\end{example}

That is, one only bets at times $n=1,4,9,16\dots$ if, for example, $\alpha = 2$. The almost sure non-bankruptcy of these strategies is an easy corollary of \cref{thm:sos-crit}. This strategy is improvable by ``betting more frequently'': the standard KT/GRAPA/aGRAPA strategy that does not withhold from betting at $n \notin I_\alpha$. The improvement makes more money under the alternative, at the price of making less money (and possible bankruptcy) under the null.

Given the subexponentially growing wealth in \cref{ex:intermittent},
it is tempting to think that exponential wealth growth $ \log W_n =  \Theta_{a.s.}(n)$ under any alternative implies null-bankruptcy. The following strategy, however, is a simple counterexample to this conjecture.

\begin{example}[Cash-GRAPA mixture] \label[example]{ex:cashgrapa} By the argument from \cref{fct:mixture-of-two}, there exists a betting strategy constituting of a 50-50 mixture between cash and GRAPA:
\begin{equation}
    W_n= \frac{1}{2} + \frac{1}{2}\prod_{k=1}^n (1 + \lgr_k(X_k - m)).
\end{equation}
Then, $W_n$ converges a.s.\ to $1/2$ under any non-degenerate null, but also grows exponentially fast under any alternative distribution, attaining the same rate $n^{-1}\log W_n = \lkelly$ as the cashless GRAPA.
\end{example}

This strategy can be improved by, again, ``removing the cash''. We note that
our ``generalized cash-removal'' idea elaborated in \eqref{sec:pf-improve} can make already exponentially powerful strategies like \cref{ex:cashgrapa} make multiplicatively more money, but cannot make subexponentially powerful strategies like \cref{ex:intermittent} make exponentially money. Generalizing the ``betting more frequently'' improvement we mentioned above for \cref{ex:intermittent} to more subexponentially powerful strategies is of interest for future work.

Next, we discuss the relevance (or lack thereof) of the Cram\'er-Rao lower bounds.
One might notice that for the null-bankrupt strategies
KT, GRAPA, and aGRAPA, the bet fractions are estimators that converge a.s.\ to a fixed number $\lKT_n \to \frac{\mu(P) - m}{C}$, $\lgr_n \to \lkelly$, and $\lagr_n \to \frac{\mu(P) - m}{\sigma^2(P) + (\mu(P)-m)^2}$ that is some functional $f(P)$ of the distribution $P$ such that $f(P) = 0$ for null $P$, and $\Exp_P( \log (1+f(P)\cdot (X-m)) ) > 0 $ for alternative $P$. For readers familiar with statistical lower bounds, it is also tempting to relate sum-of-squares criterion \cref{thm:sos-crit} or the $n^{-1/2}$ criterion \cref{cor:sq-crit} to the Cram\'er-Rao bounds. Indeed, they all concern the optimal $n^{-1/2}$ rate of regular statistical estimation problems. If a strategy's bet fractions $(\lambda_n)$ estimate one such functional $f(P)$ subject to some Cram\'er-Rao bounds, can we show the null bankruptcy of the strategy since $|\lambda_n - f(P)|$, which is $|\lambda_n |$ under the null, is at least $\approx n^{-1/2}$?

A closer look at the Cram\'er-Rao bounds suggests otherwise.  With some additional regularity assumptions on $f(P)$ and $\lambda_n$, the Cram\'er-Rao bounds lower bound the \emph{mean-square errors} of these estimators:
\begin{equation}
    \Exp_P(  (\lambda_n - f(P))^2 )  = \Omega(n^{-1}).
\end{equation}
However, it is easy to see that these variance bounds would imply neither $\lambda_n - f(P) = \Omega_{p}(n^{-1/2})$ nor $ \sum_{n=1}^\infty (\lambda_n - f(P))^2 = \infty$; a sequence of random variables may have infinite expected values while vanishing to 0 at an arbitrarily fast almost sure rate.

Finally, regarding \cref{cor:improve}, it is tempting to consider the following alternative to both the wealth lower bound condition $\bigcap_{n=1}^\infty\{ W_n \ge \rho \}$ and the predictable non-bankruptcy condition $\bigcap_{n=1}^\infty \{ \widehat W_n > \rho \}$: since these lower bounds serve to allow the leveraging (long/short mixture) argument in \cref{sec:pf-improve}, \emph{can we make $\rho$ vary predictably?} Alternatively speaking, can we work on the condition $\bigcap_{n=1}^\infty\{ W_n \ge \rho_n \}$ for a \emph{predictable sequence of wealth lower bounds $(\rho_n)$} and construct the leveraged improvement strategy mimicking what we do in \cref{sec:pf-improve}? The answer, we believe, is generally negative, and sits astride the following two areas: 
\begin{itemize}
    \item In sequential, nonnegative martingale-based inference, if $M_n(\theta)$ is a nonnegative supermartingale for every $\theta$, then so is $\int M_n(\theta) \pi(\d \theta)$ for any mixture distribution $\pi$. What about $\int M_n(\theta) \pi_n(\d \theta)$ for a \emph{predictable} sequence of mixtures $(\pi_n)$? This now requires the sequence to have \emph{consistent marginals} (see e.g.\ Section 5.1 and Lemma A.2 in \cite{flynn2023improved}, and Section 3.1 in \cite{akhavan2025bernstein}). If $\theta$ can only take two values, it is easy to see that $(\pi_n)$ can only be a constant. 
    \item In mathematical finance, one frequently encounters the notion of \emph{self-financing portfolios} (see e.g.\ Definition 5.4 in \cite{FollmerSchied+2025}). In our scenario, if one formally considers a predictable sequence of lower bounds $(\rho_n)$ on $(W_n)$ and invests in the ``$(\rho_n)$-leveraged portfolio'' with wealth process $\frac{W_n-\rho_n}{1-\rho_n}$, then one needs external capital to manage the imbalanced buying and selling of cash and $(W_n)$.
    Namely, it is not a self-financing portfolio and the wealth $\frac{W_n-\rho_n}{1-\rho_n}$ is not a martingale.
\end{itemize}

\end{document}